\documentclass{amsart}
\title{Local saturation and square everywhere}
\author{Monroe Eskew}
\address{Universität Wien \\
Institut für Mathematik \\
Kurt Gödel Research Center \\
Augasse 2-6, UZA 1 - Building 2 \\
1090 Wien \\
AUSTRIA}

\date{}
\thanks{The author is grateful to Sean Cox, Yair Hayut, Masahiro Shioya, Toshimichi Usuba, and Martin Zeman for some very helpful discussions.  The author would also like to thank the anonymous referee for their valuable suggestions, which led to significant improvements of the manuscript.}

\usepackage{color} 
\usepackage{amssymb} 
\usepackage{amsmath} 
\usepackage[utf8]{inputenc} 
\usepackage{enumerate}
\usepackage{amsthm}
\usepackage{cite}

\newtheorem{theorem}{Theorem}[section]
\newtheorem{lemma}[theorem]{Lemma}
\newtheorem{proposition}[theorem]{Proposition}
\newtheorem{corollary}[theorem]{Corollary}

\newtheorem*{definition}{Definition}
\newtheorem{remark}[theorem]{Remark}

\newtheorem{fact}[theorem]{Fact}

\newtheorem{main}{Theorem}

\DeclareMathOperator{\dom}{dom}
\DeclareMathOperator{\ran}{ran}
\DeclareMathOperator{\ot}{ot}

\DeclareMathOperator{\cf}{cf}
\DeclareMathOperator{\cof}{cof}

\DeclareMathOperator{\add}{Add}
\DeclareMathOperator{\col}{Col}
\DeclareMathOperator{\ord}{Ord}

\DeclareMathOperator{\sprt}{sprt}
\DeclareMathOperator{\crit}{crit}
\DeclareMathOperator{\id}{id}
\DeclareMathOperator{\ns}{NS}

\DeclareMathOperator{\len}{len}

\newcommand{\p}{\mathcal{P}}
\newcommand{\la}{\langle}
\newcommand{\ra}{\rangle}

\begin{document}
\maketitle

\begin{abstract}
We show that it is consistent relative to a huge cardinal that for all infinite cardinals $\kappa$, $\square_\kappa$ holds and there is a stationary $S \subseteq \kappa^+$ such that $\ns_{\kappa^+} \restriction S$ is $\kappa^{++}$-saturated.
\end{abstract}

\section{Introduction}
In his work on Suslin's problem, Jensen introduced the principles square $\square$ and diamond $\diamondsuit$ and proved that they hold everywhere in Gödel's constructible universe $L$ \cite{MR0309729}.  More specifically, in $L$, $\square_\kappa$ holds for every infinite cardinal $\kappa$, and $\diamondsuit_\kappa(S)$ holds for all regular uncountable $\kappa$ and all stationary $S \subseteq \kappa$.  A natural opposite of $\diamondsuit_\kappa(S)$ is the statement that the nonstationary ideal on $\kappa$ restricted to $S$, denoted $\ns_\kappa \restriction S$, is $\kappa^+$-saturated.  If such a stationary set $S \subseteq \kappa$ exists, we will say that $\ns_\kappa$ is \emph{locally saturated.}
While it is consistent relative to large cardinals that $\ns_\kappa$ is locally saturated for a variety of cardinals $\kappa$, Gitik and Shelah \cite{MR1363421} proved that the unrestricted $\ns_\kappa$ can never be $\kappa^+$-saturated, except in the case $\kappa = \omega_1$.  For background on saturated ideals and related topics, see \cite{MR2768692}.

Forcing $\ns_\kappa$ to be locally saturated typically results in the failure of $\square$ in the vicinity of $\kappa$ as a side-effect.  Moreover, Foreman \cite{MR2768692} observed that certain structural properties of Booelan algebras of the form $\p(\kappa^+)/I$ for $\kappa^+$-complete ideals $I$ can imply the failure of $\square_\kappa$.  A similar result was observed by Zeman in unpublished work, which we reproduce here with his permission.  Furthermore, square principles are generally opposed to very large cardinals, while saturation properties of ideals can carry significant large cardinal strength \cite{MR2768699}.  It is thus natural to ask what is the extent of the tension between these kinds of principles.  We address the situation with the following result:

\begin{main}
\label{global}
It is consistent relative to a huge cardinal that for all infinite cardinals $\kappa$, $\square_\kappa$ holds and $\ns_{\kappa^+}$ is locally saturated.
\end{main}

This improves a result of Foreman \cite{MR730584}, who proved that it is consistent relative to a huge cardinal that every successor cardinal $
\kappa$ carries a $\kappa^+$-saturated ideal.  One can check using standard arguments that in his model, $\diamondsuit_\kappa(S)$ holds for every regular $\kappa$ and every stationary $S \subseteq \kappa$.  
We suspect that $\square_\kappa$ fails for all $\kappa \geq \omega_1$ in his model because a sufficient amount of stationary reflection holds at successor cardinals.  This should follow for successors of regulars by arguments like that for Proposition \ref{proper} below, and for successors of singulars by arguments along the lines of \cite[Theorem 11.1]{MR1838355}.

The saturation of $\ns_{\omega_1}$ is equiconsistent with a Woodin cardinal; this is due to Shelah and Jensen-Steel \cite{MR3135495}.  For general successors of regular cardinals, Woodin showed in unpublished work how to force local saturation from an almost-huge cardinal, and details were given by Foreman and Komjath \cite{MR2151585}.  Because of the particulars of their construction, it was not immediately clear how to extend the result to get the nonstationary ideal to be locally saturated on several successive cardinals at once.  We overcome this technical challenge by using a crucial observation of Usuba and by building on the alternative approach to saturated ideals forged by Shioya \cite{shioyaeaston}.  In order to achieve the global result, we use Cummings' method of interleaving posets into Radin forcing \cite{MR1041044}.  An earlier version of this manuscript used a supercompact-based Radin forcing.  The advantage of Cummings' method is that since it is based on a degree of strongness rather than supercompactness, it is possible to carry out in a universe in which square holds everywhere.

The paper is organized as follows.  In Section \ref{preliminaries}, we present the essential background material on forcing and large cardinals that we will need.  In Section \ref{localsatmodule}, we present a ``modular'' version of the Foreman-Komjath construction that allows us to transform a saturated ideal on a successor cardinal into a localization of the nonstationary ideal while retaining saturation, given that the original ideal satisfies certain combinatorial properties.  In Section \ref{buildingblock}, we define a type of collapse forcing that, when combined with certain large cardinals, forces the existence of saturated ideals on successor cardinals that possess the desired combinatorial properties in a rather indestructible way.  At the end of this section, we present a new model in which the nonstationary ideal on the successor of a given regular cardinal is locally saturated, using a forcing considerably simpler than that of \cite{MR2151585}.  In Section \ref{preparation}, we construct a preparatory model in which square holds everywhere, local saturation holds at the first few successors of every Mahlo cardinal, and there exists a superstrong cardinal. Finally, in Section \ref{radinsection}, we complete the proof of Theorem \ref{global} using a version of Radin forcing that achieves the desired property at all successors of limits by ensuring that every such cardinal was a successor of a large cardinal in the preperatory model, while interleaving posets that recreate the desired situation at double successors.

\section{Preliminaries}
\label{preliminaries}

\subsection{General forcing facts}
We start by recalling some general notions and folklore results about forcing, most of which we state without proof.

A partial order $\mathbb{P}$ is said to be \emph{separative} when $p \nleq q \Rightarrow (\exists r \leq p) r \perp q$.  Every partial order $\mathbb{P}$ has a canonically associated equivalence relation $\sim_s$ and a separative quotient $\mathbb{P}_s$, which is isomorphic to $\mathbb{P}$ if $\mathbb{P}$ is already separative.  For every separative partial order $\mathbb{P}$, there is a canonical complete Boolean algebra $\mathcal{B}(\mathbb{P})$ with a dense set isomorphic to $\mathbb{P}$.

A map $e: \mathbb{P} \to \mathbb{Q}$ is an \emph{embedding} when it preserves order and incompatibility.  An embedding is said to be \emph{regular} when it preserves the maximality of antichains.  If $\mathbb P \subseteq \mathbb Q$, we say $\mathbb P$ is a \emph{regular suborder} if the identity map from $\mathbb P$ to $\mathbb Q$ is a regular embedding.  A order-preserving map $\pi : \mathbb{Q} \to \mathbb{P}$ is called a \emph{projection} when $\pi(1_\mathbb{Q}) = 1_\mathbb{P}$, and $p \leq \pi(q) \Rightarrow (\exists q' \leq q) \pi(q') \leq p$.

\begin{lemma}Suppose $\mathbb{P}$ and $\mathbb{Q}$ are partial orders. 

\begin{enumerate}
\item $G$ is a generic filter for $\mathbb{P}$ if and only if $\{ [p]_s : p \in G \}$ is a generic filter for $\mathbb{P}_s$.

\item $e : \mathbb{P} \to \mathbb{Q}$ is a regular embedding if and only if for all $q \in \mathbb{Q}$, there is $p \in \mathbb{P}$ such that for all $r \leq p$, $e(r)$ is compatible with $q$.

\item The following are equivalent:
\begin{enumerate}[(a)]
\item There is a regular embedding $e : \mathbb{P}_s \to \mathcal{B}(\mathbb{Q}_s)$.
\item There is a projection $\pi : \mathbb{Q}_s \to \mathcal{B}(\mathbb{P}_s)$.
\item There is a $\mathbb{Q}$-name $\dot{g}$ for a $\mathbb{P}$-generic filter such that for all $p \in \mathbb{P}$, there is $q \in \mathbb{Q}$ such that $q \Vdash p \in \dot{g}$.
\end{enumerate}

\item Suppose $\pi : \mathbb Q \to \mathbb P$ is a projection.  If $G$ is a filter on $\mathbb{P}$, let $\mathbb{Q}/G = \pi^{-1}[G]$.  The following are equivalent:
\begin{enumerate}[(a)]
\item $H$ is $\mathbb{Q}$-generic over $V$.
\item $G = \pi[H]$ is $\mathbb{P}$-generic over $V$, and $H$ is $\mathbb{Q} / G$-generic over $V[G]$.
\end{enumerate}


\end{enumerate}
\end{lemma}

\begin{lemma}
\label{forcingiso}
Suppose $\mathbb{P}$ and $\mathbb{Q}$ are partial orders.  $\mathcal{B}(\mathbb{P}_s) \cong \mathcal{B}(\mathbb{Q}_s)$ if and only if the following holds.  Letting $\dot{G}, \dot{H}$ be the canonical names for the generic filters for $\mathbb{P},\mathbb{Q}$ respectively, there is a $\mathbb{P}$-name for a function $\dot{f}_0$ and a $\mathbb{Q}$-name for a function $\dot{f}_1$ such that:
\begin{enumerate}
\item $\Vdash_\mathbb{P} \dot{f}_0(\dot{G})$ is a $\mathbb{Q}$-generic filter,
\item $\Vdash_\mathbb{Q} \dot{f}_1(\dot{H})$ is a $\mathbb{P}$-generic filter,
\item $\Vdash_\mathbb{P}  \dot{G} = \dot{f}_1^{\dot{f}_0(\dot{G})} (\dot{f}_0(\dot{G}))$, and $\Vdash_\mathbb{Q}  \dot{H} = \dot{f}_0^{\dot{f}_1(\dot{H})} (\dot{f}_1(\dot{H}))$.
\end{enumerate}
An isomorphism is given by $p \mapsto || \check{p} \in \dot{f}_1(\dot{H}) ||_{\mathcal{B}(\mathbb{Q}_s)}$.
\end{lemma}

A partial order $\mathbb P$ is said to be \emph{$\kappa$-closed} when any descending sequence of elements of length less than $\kappa$ has a lower bound.  A weaker property is being \emph{$\kappa$-strategically closed}, which is when the ``good'' player has a winning strategy in the following game: \emph{Bad} starts by playing some element $p_0 \in \mathbb P$, and \emph{Good} must play some $p_1 \leq p_0$.  The players alternate in choosing elements of a descending sequence, with \emph{Good} playing at limit stages.  \emph{Good} wins if a sequence of length $\kappa$ is produced, and \emph{Bad} wins if at some stage $\alpha<\kappa$, a sequence has been produced with no lower bound.  A still weaker property is being \emph{$\kappa$-distributive}, which means that the intersection of fewer than $\kappa$ dense open sets is dense.

If $\kappa < \lambda$ are ordinals, $\col(\kappa,\lambda)$ is the collection of function whose domain is a bounded subset of $\kappa$ and whose range is contained in $\lambda$, ordered by $p \leq q$ iff $p \supseteq q$.  We will use the following well-known lemma about $\kappa$-closed forcing:
\begin{lemma}
\label{folk}
If $\mathbb P$ is a $\kappa$-closed partial order that forces $|\mathbb P| = \kappa$, then $\mathcal B(\mathbb P) \cong \mathcal B(\col(\kappa,|\mathbb P|))$.
\end{lemma}


A partial order $\mathbb P$ has the \emph{$\kappa$-chain condition} ($\kappa$-c.c.) if every antichain $A \subseteq \mathbb P$ has size $<\kappa$.

\begin{lemma}[Easton]
Suppose $\mathbb P,\mathbb Q$ are partial orders, $\mathbb Q$ is $\kappa$-distributive, and $\Vdash_{\mathbb Q} \mathbb P$ is $\kappa$-c.c.  Then $\Vdash_{\mathbb P} \mathbb Q$ is $\kappa$-distributive.
\end{lemma}

\begin{proof}
Suppose $G \times H$ is $\mathbb P \times \mathbb Q$-generic, and $X$ is a sequence of orindals of length $<\kappa$ in $V[G][H]$.  Then in $V[H]$, $X$ has a $\mathbb P$-name $\tau$, and by the $\kappa$-c.c., $\tau$ can be assumed to be a subset of $V$ of size $<\kappa$.  By the distributivity of $\mathbb Q$, $\tau \in V$, so $\tau^G = X \in V[G]$. 
\end{proof}

The above lemma was crucial in Easton's proof \cite{MR0269497} that the continuum function can be ``anything reasonable.''  There, he introduced the notion of an Easton-support product, which we will use several times.  A set of ordinals $X$ is \emph{Easton} for every regular cardinal $\kappa$, $| X \cap \kappa | < \kappa$.  A collection of partial functions has \emph{Easton support} if the domain of each function in the collection is an Easton set of ordinals.

We will use several times a stronger version of the $\kappa$-c.c.\ introduced by Shelah \cite{MR1623206} which is easier to preserve under other forcings:

\begin{definition}
Let $\kappa$ be a regular cardinal and $S\subseteq\kappa$.
A partial order $\mathbb{P}$ is \emph{$S$-layered} if there is a $\subseteq$-increasing sequence of regular suborders $\langle \mathbb{Q}_\alpha : \alpha < \kappa \rangle$ such that $\mathbb{P} = \bigcup_{\alpha < \kappa} \mathbb{Q}_\alpha$, each $|\mathbb{Q}_\alpha| <\kappa$, and for some club $C$ and all $\alpha \in S \cap C$, $\mathbb{Q}_\alpha = \bigcup_{\beta<\alpha} \mathbb{Q}_\beta$.  
\end{definition}

\begin{lemma}If $S$ is a stationary subset of $\kappa$ and $\mathbb{P}$ is $S$-layered, then $\mathbb{P}$ is $\kappa$-c.c.
\end{lemma}


\begin{lemma}
\label{iterlayer}
Suppose $S \subseteq \kappa$ is stationary, $\mathbb P$ is $S$-layered, $\Vdash_{\mathbb P} \dot{\mathbb Q}$ is $\check S$-layered, and $\mathbb R$ is  regular suborder of $\mathbb P$ of size $< \kappa$.  Then: 
\begin{enumerate}
\item $\mathbb P * \dot{\mathbb Q}$ is $S$-layered.
\item $\Vdash_{\mathbb R} \mathbb P / \dot G$ is $S$-layered.
\end{enumerate}
\end{lemma}
\begin{proof}For (1), let $\langle \mathbb P_\alpha : \alpha < \kappa \rangle$ witness that $\mathbb P$ is $S$-layered, and let $\langle \dot{\mathbb Q}_\alpha : \alpha < \kappa \rangle$ be a sequence of $\mathbb P$-names for a witness to the $S$-layeredness of $\mathbb Q$.  By the $\kappa$-c.c., we may assume that we have an increasing sequence $\langle \beta_\alpha : \alpha < \kappa \rangle$ of ordinals below $\kappa$ such that for each $\alpha$, $\dot{\mathbb Q}_\alpha$ is a $\mathbb P_{\beta_\alpha}$-name.  Let $C \subseteq \kappa$ be a club such that each $\gamma$ in $C$ is closed under $\alpha \mapsto \beta_\alpha$.  Now assume $\gamma \in C \cap S$.  $\mathbb P_\gamma = \bigcup_{\alpha < \gamma} \mathbb P_\alpha$, and $\Vdash \dot{\mathbb Q}_\gamma = \bigcup_{\alpha < \gamma} \dot{\mathbb Q}_\alpha$.  Hence we may form a $\mathbb P_\gamma$-name for $\dot{\mathbb Q}_\gamma$.  It is routine to show that $\mathbb P_\gamma * \dot{\mathbb Q}_\gamma$ is a regular suborder of $\mathbb P * \dot{\mathbb Q}$.

For (2), let $\gamma$ be such that $\mathbb R \subseteq \mathbb P_\gamma$.  Then $\mathbb R$ is regular in $\mathbb P_\alpha$ for $\alpha \geq \gamma$.  If $G \subseteq \mathbb R$ is generic, then $\mathbb P_\alpha / G = \bigcup_{\gamma \leq \beta < \alpha} \mathbb P_\beta / G$ for all $\alpha \in S \setminus \gamma$.
\end{proof}

\subsection{Large cardinals and generic embeddings}
A cardinal $\kappa$ is called \emph{huge} if there is an elementary embedding $j : V \to M$ with critical point $\kappa$, where $M$ is a transitive class such that $M^{j(\kappa)} \subseteq M$.  $\kappa$ is called \emph{almost-huge} if the closure requirement of $M$ is weakened to $M^{<j(\kappa)} \subseteq M$.  $\kappa$ is called \emph{superstrong} if the requirement is weakened further to just $V_{j(\kappa)} \subseteq M$.  The value of $j(\kappa)$ in each case will be called the \emph{target}.

It is straightforward to show that $\kappa$ is huge with target $\lambda$ iff there is a normal $\kappa$-complete ultrafilter on $[\lambda]^\kappa := \{ z \subseteq \lambda : \ot(z) = \kappa \}$.  The first-order characterizations of almost-hugeness and superstrongness are more complicated, and we refer the reader to \cite{MR1994835} for details.  We will just need the following facts:

\begin{lemma}
\label{ahmin}
Suppose $\kappa$ is almost-huge with target $\lambda$.  Then there is an elementary $j : V \to M$ with the following properties:
\begin{enumerate}
\item $\crit(j) = \kappa$, $j(\kappa) = \lambda$, and $M^{<\lambda} \subseteq M$.
\item For all $x \in M$, there is an ordinal $\alpha < \lambda$ and a function $f : \alpha^{<\kappa} \to V$ such that $x = j(f)(j[\alpha])$.
\item $\sup j[\lambda] = j(\lambda) < \lambda^+$.
\item The embedding is generated by a tower of measures $T \subseteq V_\lambda$, which we will call a \emph{$(\kappa,\lambda)$-tower}.  The fact that $T$ generates an embedding with the above properties is equivalent to a first-order property in $\langle V_\lambda, \in, T \rangle$.
\end{enumerate}
\end{lemma}

\begin{lemma}
\label{ssmin}
Suppose $\kappa$ is superstrong with target $\lambda$.  Then there is an elementary $j : V \to M$ with the following properties:
\begin{enumerate}
\item $\crit(j) = \kappa$, $j(\kappa) = \lambda$, and $V_\lambda \subseteq M$.
\item For all $x \in M$, there is an $a \in [\lambda]^{<\omega}$ and a function $f : [\kappa]^{|a|} \to V$ such that $x = j(f)(a)$.
\item If $\cf(\lambda)>\kappa$, then $M^\kappa \subseteq M$.
\item The embedding is generated by a $(\kappa,\lambda)$-extender $E$.  The fact that $E$ generates an embedding with the above properties is equivalent to a first-order property in $\langle V_\lambda, \in, E \rangle$.
\end{enumerate}
\end{lemma}

Silver observed that if $j : M \to N$ is an elementary embedding between models of set theory, $\mathbb P \in M$ is a partial order, and $G$ is $\mathbb P$-generic over $M$, then $j$ can be extended to an elementary embedding with domain $M[G]$ if and only if we can find a filter $\hat G$ that is $j(\mathbb P)$-generic over $N$, with $j[G] \subseteq \hat G$.  We now describe some general situations in which almost-huge and superstrong embeddings can be generically extended.

\begin{definition}
A partial order $\mathbb Q$ is \emph{$(\kappa,\lambda)$-nice} when there is a sequence $\langle \mathbb Q_\alpha : \alpha \leq \lambda \rangle$ of regular suborders such that:
\begin{enumerate}
\item The sequence is $\subseteq$-increasing,  $\bigcup_{\alpha < \lambda} \mathbb Q_\alpha = \mathbb Q_\lambda = \mathbb Q$, and for all $\alpha < \lambda$, $| \mathbb Q_\alpha | < \lambda$.
\item For each $\alpha \leq \lambda$, any two compatible elements of $\mathbb Q_\alpha$ have an infimum in $\mathbb Q_\alpha$, and every directed subset of $\mathbb Q_\alpha$ of size $<\kappa$ has an infimum in $\mathbb Q_\alpha$.
\end{enumerate}
\end{definition}

\begin{lemma}
\label{ahext}
Suppose the following:
\begin{enumerate}
\item $j : V \to M$ is an almost-huge embedding derived from a $(\kappa,\lambda)$-tower.
\item $\mathbb P \subseteq V_\kappa$ is a partial order, and $j(\mathbb P)$ is $\lambda$-c.c.\ in $V$.
\item $\dot{\mathbb Q}$ is a $\mathbb{P}$-name for a $(\kappa,\lambda)$-nice partial order.
\item There is a projection $\pi : j(\mathbb P) \to \mathbb P * \dot{\mathbb Q}$ such that for $p \in \mathbb P$, $\pi(p) = (p,\dot 1)$.
\end{enumerate}
If $\hat G$ is $j(\mathbb P)$-generic over $V$ and $G * H = \pi[\hat G]$, then in $V[\hat G]$ we can extend $j$ to $j : V[G * H] \to M[\hat G * \hat H]$, such that $M[\hat G * \hat H]^{<\lambda} \cap V[\hat G] \subseteq M[\hat G * \hat H]$.
\end{lemma}

\begin{proof}
Let $\hat G \subseteq j(\mathbb P)$ be generic, and let $G * H = \pi[\hat G]$.  Since $\pi$ and $j$ are the identity on $\mathbb P$, we can extend to $j : V[G] \to M[\hat G]$.  By the $\lambda$-c.c. and the closure of $M$, $\ord^{<\lambda} \cap V[\hat G] \subseteq M[\hat G]$.  We will build the desired $\hat H$ in $V[\hat G]$, so we will get $\ord^{<\lambda} \cap V[\hat G] \subseteq M[\hat G * \hat H]$ as well.

Let $\langle \mathbb Q_\alpha : \alpha \leq \lambda \rangle$ witness the $(\kappa,\lambda)$-niceness of $\mathbb Q$ in $V[G]$. For each $\alpha < \lambda$, let $H_\alpha = H \cap \mathbb Q_\alpha$, and let $m_\alpha = \inf j[H_\alpha]$, which exists because $j[H_\alpha]$ is an element of $M[\hat G]$ and a directed subset of $j(\mathbb Q_\alpha)$ of size $<j(\kappa)$.  
Let us observe the following: If $\alpha<\beta<\lambda$ and $q \leq m_\alpha$ is in $j(\mathbb Q_\alpha)$, then $q$ is compatible with $m_\beta$.  To show this, note that for any $r \in \mathbb Q_\beta$, set $D_r = \{ p \in \mathbb Q_\alpha : p \perp r \text{ or } (\forall p' \leq p)p' \not\perp r \}$ is dense in $\mathbb Q_\alpha$.  Suppose towards a contradiction that $q \leq m_\alpha$ is in $j(\mathbb Q_\alpha)$ and $q \perp m_\beta$.  Then by the directed closure, $q \perp j(r)$ for some $r \in H_\beta$.  But there is $p \in D_r \cap H_\alpha$, and $p \not\perp r$.  Since $q \leq m_\alpha \leq j(p)$, by elementarity $q$ is compatible with $j(r)$, a contradiction.

Let $\la A_\alpha : \alpha < \lambda \ra$ list in $V[\hat G]$ the maximal antichains of $j(\mathbb Q)$ that live in $M[\hat G]$.  We inductively build a filter $\hat H$ that is $j(\mathbb Q)$-generic over $M[\hat G]$, and contains each $m_\alpha$.  This will guarantee $j[H] \subseteq \hat H$, and thus by Silver's criterion we will be done.  Choose an increasing sequence of ordinals $\la \beta_\alpha : \alpha < \lambda \ra$ such that $A_\alpha \subseteq j(\mathbb Q_{\beta_\alpha})$.  Find some $a_0 \in A_0$ such that $m_{\beta_0}$ is compatible with $a_0$, and let $q_0 = m_{\beta_0} \wedge a_0$.  Suppose inductively that for some $\gamma < \lambda$, we have constructed a descending sequence $\la q_\alpha : \alpha < \gamma \ra$ such that for each $\alpha$, $q_\alpha \in j(\mathbb Q_{\beta_\alpha})$ and $q_\alpha \leq a_\alpha \wedge m_{\beta_\alpha}$ for some $a_\alpha \in A_\alpha$.  By the observation of the previous paragraph, $m_{\beta_\gamma}$ is compatible with $q_\alpha$ for all $\alpha<\gamma$.  Hence the directed set $\{ m_{\beta_\gamma} \wedge q_\alpha : \alpha < \gamma \}$ has an infimum $q_\gamma'$.  Let $q_\gamma = q'_\gamma \wedge a_\gamma$ for some $a_\gamma \in A_\gamma$.  This completes the induction.
\end{proof}

\begin{lemma}
\label{ssext}
Suppose the following:
\begin{enumerate}
\item $j : V \to M$ is an superstrong embedding derived from a $(\kappa,\lambda)$-extender with $\lambda$ inaccessible.
\item $\mathbb P \subseteq V_\kappa$ is a partial order, and $j(\mathbb P)$ is $\lambda$-c.c.\ in $V$.
\item $\mathbb P$ forces that the quotient $j(\mathbb P)/G$ is $\kappa^+$-distributive.
\end{enumerate}
If $\hat G$ is $j(\mathbb P)$-generic over $V$, then $\kappa$ is superstrong with target $\lambda$ in $V[\hat G]$.
\end{lemma}

\begin{proof}
Since $\mathbb P$ is $\kappa$-c.c., $j(A) = A$ for every maximal antichain $A \subseteq \mathbb P$, so $\mathbb P$ is a regular suborder of $j(\mathbb P)$.  Let $\alpha < \kappa$.  Since every subset of $\alpha$ added by $j(\mathbb P)$ is added by $\mathbb P$, reflection gives that there is a regular suborder $\mathbb Q \subseteq \mathbb P$ of size $<\kappa$ that adds all subsets of $\alpha$.  Thus $M$ satisfies that for every $\alpha < \lambda$, there is a regular suborder $\mathbb Q \subseteq j(\mathbb P)$ of size $<\lambda$ that adds all subsets of $\alpha$.  This is true in $V$ as well, since by the $\lambda$-c.c., any $j(\mathbb P)$-name $\tau$ for a subset of $\alpha$ is equivalent to an $\mathbb Q'$-name $\tau'$ for some regular suborder $\mathbb Q' \in V_\lambda \subseteq M$, and $\tau'$ must be equivalent to a $\mathbb Q$-name by what $M$ thinks of $\mathbb Q$.  Thus $\lambda$ remains inaccessible after forcing with $j(\mathbb P)$.
 
Let $\hat G$ be $j(\mathbb P)$-generic over $V$, and let $G = \hat G \cap \mathbb P$.  We can extend $j$ to $j : V[G] \to M[\hat G]$.  By the $\lambda$-c.c., every element of $(V_\lambda)^{V[\hat G]}$ is $\tau^{\hat G}$ for some $j(\mathbb P)$-name $\tau \in (V_\lambda)^V$.  Since $V_\lambda \subseteq M$, $(V_\lambda)^{M[\hat G]} = (V_\lambda)^{V[\hat G]}$.

For every $x \in M[\hat G]$, there is a $j(\mathbb P)$-name $\tau$ such that $x = \tau^{\hat G}$, and there is an $a \in [\lambda]^{<\omega}$ and a function $f$ with domain $[\kappa]^{|a|}$ in $V$ such that $\tau = j(f)(a)$.  We may assume that for every $b \in \dom f$, $f(b)$ is a $\mathbb P$-name.  If we define a function $g$ in $V[G]$ by $g(b) = f(b)^G$, then we have $x = j(g)(a)$.

Let $\mathbb Q$ be the quotient $j(\mathbb P)/G$ in $V[G]$, and let us write $\hat G$ as $G * H \subseteq \mathbb P * \dot{\mathbb Q}$.  Let $D \in M[\hat G]$ be dense open subset of $j(\mathbb Q)$.  Let $a,g$ be a such that $a \in [\lambda]^{<\omega}$, $g \in V[G]$ is a function with domain $[\kappa]^{|a|}$, and $D = j(g)(a)$.  We may assume that $g(b)$ is a dense open subset of $\mathbb P$ for all $b \in \dom g$.  Let $E = \bigcap_{b \in \dom g} g(b)$.  By the distributivity of $\mathbb Q$, $E$ is dense.  Thus there is $q \in E \cap H$.  Since $E \subseteq g(b)$ for all $b \in \dom g$, $j(q) \in j(E) \subseteq D$.  Therefore the image $j[H]$ generates a filter $\hat H$ which is $j(\mathbb Q)$-generic over $M[\hat G]$. We may extend the embedding to $j : V[G * H] \to M[\hat G * \hat H]$.  Since $(V_\lambda)^{V[\hat G]} = (V_\lambda)^{M[\hat G]} \subseteq (V_\lambda)^{M[\hat G * \hat H]} \subseteq (V_\lambda)^{V[\hat G]}$, $\kappa$ is superstrong with target $\lambda$ in $V[\hat G]$.
\end{proof}

\subsection{Ideals and duality}
We will be interested in extending embeddings that arise from forcing with Boolean algebras of the form $\p(Z)/I$, where $I$ is an ideal over $Z$, and in computing what happens to this algebra in generic extensions.  To this end, we present an optimal generalization of a result of Foreman \cite{MR3038554} on this topic from the author's thesis \cite{MR3279214}.  Let us first review some basic facts concerning ideals, which can be found in \cite{MR2768692}.  An ideal $I$ over a set $Z$, gives a notion of smallness or ``$I$-measure-zero'' for subsets of $Z$.  We will sometimes refer to the family $\p(Z) \setminus I$ as $I^+$ or the ``$I$-positive sets,'' and refer to the family $\{ Z \setminus A : A \in I \}$ as $I^*$ or the ``$I$-measure-one sets.''  Recall that an ideal $I$ over a set $Z$ is called \emph{precipitous} if whenever $G \subseteq \p(Z)/I$ is generic, then the ultrapower $V^Z/G$ is well-founded.

\begin{fact}
Let $I$ be an ideal over $Z\subseteq \p(\lambda)$.  Suppose that $I$ is normal and $\lambda^+$-saturated.  Then $I$ is precipitous, and whenever $j : V \to M \subseteq V[G]$ is a generic ultrapower embedding arising from $I$, then $M^\lambda \cap V[G] \subseteq M$.  Furthermore, if $\kappa = \mu^+$, then $\{ z \in Z : \cf(\sup z) = \cf(\mu) \} \in I^*$.
\end{fact}

%

\begin{theorem}
\label{dualitygen}
Suppose $I$ is a precipitous ideal on $Z$ and $\mathbb{P}$ is a Boolean algebra.  Let $j: V \to M \subseteq V[G]$ denote a generic ultrapower embedding arising from $I$.  Suppose $\dot{K}$ is a $\p(Z)/I$-name for an ideal on $j(\mathbb P)$ such that whenever $G * h$ is $\p(Z)/I * j(\mathbb P)/ \dot K$-generic and $\hat H = \{ p : [p]_K \in h \}$, we have:

\begin{enumerate}
\item $1 \Vdash_{\p(Z)/I * j(\mathbb P)/ \dot K} \hat{H}$ is $j(\mathbb{P})$-generic over $M$,
\item $1 \Vdash_{\p(Z)/I * j(\mathbb P)/ \dot K} j^{-1}[\hat{H}]$ is $\mathbb{P}$-generic over $V$, and
\item for all $p \in \mathbb{P}$, $1 \nVdash_{\p(Z)/I} j(p) \in \dot K$.
\end{enumerate}
Then there is $\mathbb{P}$-name $\dot J$ for an ideal on $Z$ and a canonical isomorphism
\[ \iota : \mathcal{B}( \mathbb{P} * \p(Z)/\dot J) \cong \mathcal{B}( \p(Z)/I * j(\mathbb{P})/\dot K ). \]
\end{theorem}
\begin{proof}
Let $e : \mathbb{P} \to \mathcal{B}(\p(Z)/I * j(\mathbb{P})/\dot K)$ be defined by $p \mapsto || j(p) \in \hat{H} ||$.  By (3), this map has trivial kernel.  By elementarity, it is an order and antichain preserving map.  If $A \subseteq \mathbb{P}$ is a maximal antichain, then it is forced that $j^{-1}[\hat{H}] \cap A \not= \emptyset$.  Thus $e$ is regular.

Whenever $H \subseteq \mathbb{P}$ is generic, there is a further forcing yielding a generic $G * h \subseteq \p(Z)/I * j(\mathbb{P})/\dot K$ such that $j[H] \subseteq \hat{H}$.  Thus there is an embedding $\hat{j} : V[H] \to M[\hat{H}]$ extending $j$.  In $V[H]$, let $J = \{ A \subseteq Z : 1 \Vdash_{(\p(Z)/I * j(\mathbb{P})/\dot K) / e[H]} [\id]_M \notin \hat{j}(A) \}$.  In $V$, define a map $\iota : \mathbb{P} * \p(Z)/\dot J \to \mathcal{B}( \p(Z)/I * j(\mathbb{P})/\dot K)$ by $(p,\dot{A}) \mapsto e(p) \wedge || [\id]_M \in \hat{j}(\dot{A}) ||$.  It is easy to check that $\iota$ is order and antichain preserving.

We want to show the range of $\iota$ is dense.  Let $(B,\dot{q}) \in \p(Z)/I * j(\mathbb{P})/\dot K$. Without loss of generality, there is some $f : Z \to V$ in $V$ such that $B \Vdash \dot{q} = [[f]_M]_K$.  By the regularity of $e$, let $p \in \mathbb{P}$ be such that for all $p' \leq p$, $e(p') \wedge (B,\dot{q}) \not= 0$.  Let $\dot{A}$ be a $\mathbb{P}$-name such that $p \Vdash \dot{A} = \{ z \in B : f(z) \in \dot H \}$, and $\neg p \Vdash \dot{A} = Z$.  $1 \Vdash_\mathbb{P} \dot{A} \in J^+$ because for any $p' \leq p$, we can take a generic $G * h$ such that $e(p') \wedge (B, \dot{q}) \in G*h$.  Here we have $[\id]_M \in j(B)$ and $[f]_M \in \hat{H}$, so $[\id]_M \in \hat{j}(A)$.  Furthermore, $\iota(p,\dot{A})$ forces $B \in G$ and $q \in h$, showing $\iota$ is a dense embedding.
\end{proof}

\begin{proposition}
\label{ultequal}
If $Z,I,\mathbb{P},\dot J,\dot K,\iota$ are as in Theorem~\ref{dualitygen}, then whenever $H \subseteq \mathbb{P}$ is generic, $J$ is precipitous and has the same completeness and normality that $I$ has in $V$.  Also, if $\bar{G} \subseteq \p(Z)/ J$ is generic and $G * h = \iota[H * \bar{G}]$, then if $\hat{j} : V[H] \to M[\hat{H}]$ is as above, $M[\hat{H}] = V[H]^Z/\bar{G}$ and $\hat{j}$ is the canonical ultrapower embedding.
\end{proposition}
\begin{proof}
Suppose $H * \bar{G} \subseteq  \mathbb{P} * \p(Z) / \dot J$ is generic, and let $G * h = \iota[H * \bar{G}]$ and $\hat{H} = \{ p : [p]_K \in h \}$.  For $A \in J^+$, $A \in \bar{G}$ if and only if $[\id]_M \in \hat{j}(A)$.  If $i : V[H] \to N = V[H]^Z / \bar{G}$ is the canonical ultrapower embedding, then there is an elementary embedding $k : N \to M[\hat{H}]$ given by $k([f]_N) = \hat{j}(f)([\id]_M)$, and $\hat{j} = k \circ i$.  Thus $N$ is well-founded, so $J$ is precipitous.  If $f : Z \to \ord$ is a function in $V$, then $k([f]_N) = j(f)([\id]_M) = [f]_M$.  Thus $k$ is surjective on ordinals, so it must be the identity, and $N = M[\hat{H}]$.  Since $i = \hat{j}$ and $\hat{j}$ extends $j$, $i$ and $j$ have the same critical point, so the completeness of $J$ is the same as that of $I$.  Finally, since $[\id]_N = [\id]_M$, $I$ is normal in $V$ if and only if $J$ is normal in $V[H]$, because $j \restriction \bigcup Z = \hat{j} \restriction \bigcup Z$, and normality is equivalent to $[\id] = j[\bigcup Z]$.
\end{proof}

Theorem~\ref{dualitygen} is optimal in the sense that it characterizes exactly when an elementary embedding coming from a precipitous ideal can have its domain enlarged via forcing:

\begin{proposition}
Let $I$ be a precipitous ideal on $Z$ and $\mathbb{P}$ a Boolean algebra.  The following are equivalent.
\begin{enumerate}
\item In some generic extension of a $\p(Z)/I$-generic extension, there is an elementary embedding $\hat j : V[H] \to M[\hat H]$, where $j : V \to M$ is the elementary embedding arising from $I$ and $H$ is  $\mathbb P$-generic over $V$.
\item There are  $p \in \mathbb P$, $A \in I^+$, and a $\p(A)/I$-name $\dot K$ for an ideal on $j(\mathbb P \restriction p)$ such that $\p(A)/I * j(\mathbb P \restriction p)/ \dot K$ satisfies the hypothesis of Theorem \ref{dualitygen}.
\end{enumerate}
\end{proposition}
\begin{proof}
$(2) \Rightarrow (1)$ is trivial.  To show $(1) \Rightarrow (2)$, let $\dot{\mathbb Q}$ be a $\p(Z)/I$-name for a partial order, and suppose $A \in I^+$ and $\dot{H}_0$ are such that $\Vdash_{\p(A)/I * \dot{\mathbb Q}}$ ``$\dot{H}_0$ is $j(\mathbb P)$-generic over $M$ and $j^{-1}[\dot{H}_0]$ is $\mathbb P$-generic over $V$.''  By the genericity of $j^{-1}[\dot{H}_0]$, the set of $p \in \mathbb P$ such that $\Vdash_{\p(A)/I * \dot{\mathbb Q}} j(p) \notin \dot{H}_0$ is not dense.  So let $p_0$ be such that for all $p \leq p_0$, $|| j(p) \in \dot{H}_0 || \not = 0$.  In $V^{\p(A)/I}$, define an ideal $K$ on $j(\mathbb P \restriction p_0)$ by $K = \{ p \in j(\mathbb P \restriction p_0) : 1 \Vdash_\mathbb{Q} p \notin \dot{H}_0 \}$.  We claim $K$ satisfies the hypotheses of Theorem \ref{dualitygen}.  Let $G * h$ be $\p(A)/I * j(\mathbb P \restriction p_0)/\dot K$-generic.  In $V[G*h]$, let $\hat{H} = \{ p \in j(\mathbb{P} \restriction p_0) : [p]_K \in h \}$.

\begin{enumerate}
\item If $D \in M$ is open and dense in $j(\mathbb{P} \restriction p_0)$, then $\{ [d]_K : d \in D$ and $d \notin K \}$ is dense in $j(\mathbb{P} \restriction p_0)/K$.  For otherwise, there is $p \in j(\mathbb{P} \restriction p_0) \setminus K$ such that $p \wedge d \in K$ for all $d \in D$.  By the definition of $K$, we can force with $\mathbb{Q}$ over $V[G]$ to obtain an $M$-generic filter $H_0 \subseteq j(\mathbb{P})$ with $p \in H_0$.  But $H_0$ cannot contain any elements of $D$, so it is not generic over $M$, a contradiction.  Thus if $h \subseteq j(\mathbb{P} \restriction p_0)/K$ is generic over $V[G]$, then $\hat{H}$ is $j(\mathbb{P} \restriction p_0)$-generic over $M$.
\item If $\mathcal A \in V$ is a maximal antichain in $\mathbb{P} \restriction p_0$, then $\{ [j(a)]_K : a \in \mathcal A$ and $j(a) \notin K \}$ is a maximal antichain in $j(\mathbb{P}\restriction p_0)/K$.  For otherwise, there is $p \in j(\mathbb{P}\restriction p_0) \setminus K$ such that $p \wedge j(a) \in K$ for all $a \in \mathcal A$.  We can force with $\mathbb{Q}$ over $V[G]$ to obtain a filter $H_0 \subseteq j(\mathbb{P})$ with $p \in H_0$.  But $H_0$ cannot contain any elements of $j[\mathcal A]$, so $j^{-1}[H_0]$ is not generic over $V$, a contradiction.
\item If $p \in \mathbb{P} \restriction p_0$, $|| j(p) \in \dot{H}_0 ||_{\p(A)/I * \dot{\mathbb Q}} \not=0$, so $1 \nVdash_{\p(Z)/I} j(p) \in \dot K$.
\end{enumerate}
\end{proof}

\begin{lemma}
\label{generated}
Suppose the ideal $K$ in Theorem~\ref{dualitygen} is forced to be principal.  Let $\dot{m}$ be such that $\Vdash_{\p(Z)/I} \dot{K} = \{ p \in j(\mathbb{P}) : p \leq \neg \dot{m} \}$.  Suppose $f$ and $A$ are such that $A \Vdash \dot{m} = [f]$, and $\dot{B}$ is a $\mathbb{P}$-name for $\{ z \in A : f(z) \in H \}$.  Let $\bar{I}$ be the ideal generated by $I$ in $V[H]$.  Then $\bar{I} \restriction B = J \restriction B$, where $J$ is given by Theorem~\ref{dualitygen}.
\end{lemma}

\begin{proof}
Clearly $J \supseteq \bar{I}$.  Suppose that $p_0 \Vdash$ ``$\dot{C} \subseteq \dot{B}$ and $\dot{C} \in \bar{I}^+$,'' and let $p_1 \leq p_0$ be arbitrary.  Without loss of generality, $\mathbb{P}$ is a complete Boolean algebra.  For each $z \in Z$, let $b_z = || z \in \dot{C} ||$.  In $V$, define $C' = \{ z : p_1 \wedge b_z \wedge f(z) \not= 0 \}$.  $p_1 \Vdash \dot{C} \subseteq C'$, so $C' \in I^+$.  If $G \subseteq \p(Z)/I$ is generic with $C' \in G$, then $j(p_1) \wedge b_{[\id]} \wedge \dot{m} \not= 0$.  Take $\hat{H} \subseteq j(\mathbb{P})$ generic over $V[G]$ with $j(p_1) \wedge b_{[\id]} \wedge \dot{m} \in \hat{H}$.  Since $b_{[\id]} \Vdash^M_{j(\mathbb{P})} [\id] \in \hat{j}(C)$, $p_1 \nVdash \dot{C} \in \dot J$ as $p_1 \in H = j^{-1}[\hat{H}]$.  Thus $p_0 \Vdash \dot C \in \dot J^+$.
\end{proof}

\begin{corollary}
\label{dualitynicecase}
If $I$ is a $\kappa$-complete precipitous ideal on $Z$ and $\mathbb{P}$ is $\kappa$-c.c.,\ then there is a canonical isomorphism $\iota : \mathbb{P} * \p(Z)/ \bar{I} \cong \p(Z)/I * j(\mathbb{P})$.
\end{corollary}
\begin{proof}
If $G * \hat{H} \subseteq \p(Z)/I * j(\mathbb{P})$ is generic, then for any maximal antichain $A \subseteq \mathbb{P}$ in $V$, $j[A] = j(A)$, and $M \models j(A)$ is a maximal antichain in $j(\mathbb{P})$.  Thus $j^{-1}[\hat{H}]$ is $\mathbb{P}$-generic over $V$, and clearly for each $p \in \mathbb{P}$, we can take $\hat{H}$ with $j(p) \in \hat{H}$.  Taking a $\p(Z)/I$-name $\dot K$ for the trivial ideal on $j(\mathbb P)$, Theorem~\ref{dualitygen} implies that there is a $\mathbb P$-name $\dot{J}$ for an ideal on $Z$ and an isomorphism $\iota : \mathcal{B}(\mathbb{P} * \p(Z) / \dot J) \to \mathcal{B}(\p(Z)/I * j(\mathbb{P}))$, and Lemma~\ref{generated} implies that $\Vdash_{\mathbb P} \dot J = \bar{I}$.
\end{proof}

\section{A local saturation module}
\label{localsatmodule}
In this section, we show how to transform saturated ideals with certain properties into a restriction of the nonstationary ideal, while retaining saturation.  Some key ideas are taken from \cite{MR2151585}.  Given a set of ordinals $S$, let $\mathbb C(S)$ denote the forcing for shooting a club through $\sup(S)$ by initial segments.

\begin{lemma}
\label{itershoot}Assume GCH, $\mu$ is regular, $\kappa = \mu^+$, and $S \subseteq \kappa \cap \cof(\mu)$ is stationary.  Let $\langle \mathbb P_\alpha,\dot{\mathbb Q}_\beta : \alpha\leq \lambda$, $\beta < \lambda \rangle$ be an iteration with ${<}\kappa$-supports such that for each $\alpha$, there is a $\mathbb P_\alpha$-name $\dot S_\alpha$ for a subset of $\kappa$ such that $\Vdash_{\mathbb P_\alpha} \dot{\mathbb Q}_\alpha = \mathbb C(\dot S_\alpha \cup \check S \cup \cof({<}\mu))$.  Then:
\begin{enumerate}
\item $\mathbb P_\lambda$ is $\mu$-closed.
\item $\mathbb P_\lambda$ is $\kappa$-distributive.
\item $\mathbb P_\lambda$ preserves every stationary $T \subseteq S$.
\item The set $\overline{\mathbb P_\lambda} = \{ p \in \mathbb P_\lambda : (\forall \alpha < \lambda)(\exists r \subseteq \kappa) p{\restriction} \alpha \Vdash_\alpha p(\alpha) = \check r \}$, is dense in $\mathbb P_\lambda$.
\end{enumerate}
\end{lemma}

\begin{proof}
(1) is easy.  For (2) and (3), fix a stationary $T \subseteq S$, let $p \in \mathbb P_\lambda$, and let $\dot f$ be a name for a function from $\mu$ to the ordinals.  Let $\theta$ be a large enough regular cardinal, and let $N \prec H_\theta$ be elementary such that $N^{<\mu} \subseteq N$, $N \cap \kappa \in T$, $| N | = \mu$, and $p, \dot{f}, \mathbb P_\lambda \in N$.  List the dense open subsets of $\mathbb P_\lambda$ in $N$ as $\langle D_\alpha : \alpha < \mu \rangle$.  Note that for all $\alpha < \kappa$, the set of $q$ such that for all $\beta \in \sprt(q)$, $q \restriction \beta \Vdash \sup q(\beta) > \check\alpha$ is dense.
Note also that for all $q \in \mathbb P_\lambda \cap N$, $\sprt(q) \subseteq N$.
Using $\mu$-closure, build a descending chain $\langle q_\alpha : \alpha < \mu \rangle \subseteq N$ below $p$ such that each $q_\alpha \in D_\alpha$.  
Let $q$ be a function with domain $N \cap \lambda$ such that for all $\beta$, $q(\beta)$ is the canonical $\mathbb P_\beta$-name for $\bigcup_{\alpha < \mu}q_\alpha(\beta) \cup \{ N \cap \kappa \}$.  By induction we see that $q$ is a condition in $\mathbb P_\lambda$ below each $q_\alpha$:  If $q \restriction \beta$ is a condition below each $q_\alpha \restriction \beta$, then $q \restriction \beta \Vdash_{\beta}$ ``$\langle q_\alpha(\beta) : \alpha < \mu \rangle$ is a chain of bounded closed subsets of $(\dot S_\beta \cup \check S \cup \cof({<}\mu)) \cap N \cap \kappa$ ordered by end-extension, and $\{ \sup q_\alpha(\beta) : \alpha < \mu \}$ is unbounded in $N \cap \kappa$.''  Hence $q {\restriction} \beta \Vdash q(\beta) \leq q_\alpha(\beta)$ for all $\alpha < \mu$.  Limit steps are trivial.  Thus $q$ decides $\dot f$ and forces $T \cap \dot C \not= \emptyset$.


For (4), we proceed by induction on $\lambda$.  Suppose $\lambda = \beta +1$ and the result holds for $\mathbb P_\beta$.  If $p \in \mathbb P_\lambda$, then by $\kappa$-distributivity we may extend $p \restriction \beta$ to some $q \in \overline{\mathbb P_\beta}$ such that $q \Vdash p(\beta) = \check r$ for some $r \subseteq \kappa$.  If $\cf(\lambda) = \delta < \mu$, choose an increasing sequence $\langle \lambda_i  : i<\delta \rangle$ cofinal in $\lambda$.  Let $p \in \mathbb P_\lambda$, and build a descending chain $\langle p_i : i < \delta \rangle$ below $p$ such that each $p_i \restriction i \in \overline{\mathbb P_{\lambda_i}}$.  The function $p'$ such that for all $\alpha \in \bigcup_{i<\delta} \sprt(p_i)$, $p'(\alpha)$ is a name for $\bigcup_{i<\delta} p_i(\alpha) \cup \sup(\bigcup_{i<\delta} p_i(\alpha))$ is a condition below $p$ in $\overline{\mathbb P_\lambda}$.  If $\cf(\lambda) > \mu$, then for any $p \in \mathbb P_\lambda$, $p \in \mathbb P_\beta$ for some $\beta < \lambda$, so we just apply the induction hypothesis.  Finally, if $\cf(\lambda) = \mu$, then given a $p \in \mathbb P_\lambda$, we take an elementary substructure $N$ with $p \in N$ as in the previous claims.  Using the induction hypothesis, we build a descending chain of length $\mu$ of elements below $p$ contained in $\overline{\mathbb P_\lambda}$, as in the case $\cf(\lambda)<\mu$. We then construct a master condition $q$ below this chain as in the previous claims, which will be in $\overline{\mathbb P_\lambda}$.
\end{proof}

In a context like above where the fixed objects $\kappa,\mu,S$ are clear, we will abbreviate the forcing $\mathbb C(T \cup S \cup \cof({<}\mu))$ by $\mathbb C(T)$.  An iteration of such forcings with ${<}\kappa$-support will be called an \emph{$S$-iteration}. 

\begin{lemma}[Foreman-Komjath]
\label{absorbshoot}
Assume GCH, $\mu$ is regular, $\kappa = \mu^+$, and $S \subseteq \kappa \cap \cof(\mu)$ is stationary.  Let $\la \mathbb P_\alpha,\dot{\mathbb Q}_\alpha : \alpha < \lambda \ra$ be an $S$-iteration.  There is a sequence $\la \pi_\alpha : \alpha \leq \lambda \ra$ such that:
\begin{enumerate}
\item $\pi_\alpha : \mathcal B(\mathbb C(\emptyset) \times \add(\kappa,\alpha)) \to \mathcal B(\mathbb C(\emptyset) \times \mathbb P_\alpha)$ is a projection.
\item For $\alpha < \beta \leq \lambda$, $\pi_\alpha = \pi_\beta \restriction \mathcal B(\mathbb C(\emptyset) \times \add(\kappa,\alpha))$.
\end{enumerate}
\end{lemma}

\begin{proof}
It suffices to show that, after forcing with $\mathbb C(\emptyset)$, which makes $S$ contain almost all ordinals of cofinality $\mu$, there exists a system of projections that commutes as desired, defined on a dense subset of $\add(\kappa,\lambda)$ and mapping into $\overline{\mathbb P_\lambda}$.  Note that since $\mathbb C(\emptyset)$ adds no new $<\kappa$-sequences, if $G \subseteq \mathbb C(\emptyset)$ is generic, then $\add(\kappa,\lambda)^V =\add(\kappa,\lambda)^{V[G]}$, and if $\mathbb P_\lambda^{V[G]}$ is $S$-iteration defined in $V[G]$ using the same sequence of names for subsets of $\kappa$, then we see inductively that $\overline{\mathbb P_\lambda}^V = \overline{\mathbb P_\lambda}^{V[G]}$.

Let us work in a generic extension by $\mathbb C(\emptyset)$, and let $C \subseteq \kappa$ be the generic club.  Note that in $V[C]$, there is a $\kappa$-closed dense subset of $\overline{\mathbb P_\lambda}$; namely the set of those $p$ such that for all $\alpha \in \sprt(p)$, $\max p(\alpha) \in C$.

Say a condition $p \in \add(\kappa,\lambda)$ is \emph{flat} if $\dom p = X \times \xi$ for some $X \subseteq \lambda$ and some $\xi < \kappa$.  Clearly, the set of flat conditions is dense.
Fix some bijection $f : \kappa \to [\kappa]^{<\kappa}$ such that $f(0) = \emptyset$.  We will define the map $\pi_\lambda$, and it will be clear from the construction that the same definition can be run at a different $\lambda'$, and the desired coherence condition (2) above will hold.

Let $X \in [\lambda]^{<\kappa}$, and let $p : X \to [\kappa]^{<\kappa}$.  Let $q \in \overline{\mathbb P_\lambda}$, with $\sprt(q) = X$.  Inductively define $q \wedge p \in \overline{\mathbb P_\lambda}$ on $\alpha\in X$ by putting $(q \wedge p)(\alpha) = p(\alpha)$ if:
\begin{enumerate}
\item $p(\alpha)$ is a closed bounded set end-extending $q(\alpha)$.
\item $\max p(\alpha) \in C$.
\item $(q \wedge p)\restriction \alpha \Vdash_\beta \check p(\alpha) \in \dot{\mathbb Q}_\alpha$.
\end{enumerate}
Define $(q \wedge p)(\alpha) = q(\alpha)$ otherwise.  
Let $p \in \add(\kappa,\lambda)$ be a flat condition with domain $X \times \xi$.  For $i < \xi$, let $p_i = p \restriction (X \times i)$; we will define $\pi(p)$ as a limit of the $\pi(p_i)$.  Let $\pi(p_0) = 1_{\mathbb P_\alpha}$.  Given $\pi(p_i)$, consider the sequence $q_i = \langle f(p(\beta,i)) : \beta \in X \rangle$.  Let $\pi(p_{i+1}) = \pi(p_i) \wedge q_i$.  At limit $j$, we take $\pi(p_j) = \inf_{i<j} \pi(p_i)$.

For $p,q \in \overline{\mathbb P_\lambda}$, let us say that $q$ is a \emph{horizontal extension} of $p$ if $q(\alpha) = p(\alpha)$ for $\alpha \in \sprt(p)$.  Call a flat condition $p \in \mathbb \add(\kappa,\lambda)$ with domain $X \times \xi$ \emph{good} if for every $i< \xi$ and every $\alpha \in X$, if $\pi(p_{i+1}) \restriction \alpha$ has a horizontal extension in $\overline{\mathbb P_\alpha}$ deciding whether $f(p(\alpha,i)) \in \dot{\mathbb Q}_\alpha$, then already $\pi(p_{i+1}) \restriction \alpha$ decides this.  We claim that the set of good conditions is dense, and that $\pi$ restricted to this set is a projection.

Let $p \in \add(\kappa,\lambda)$ be flat with domain $X_0 \times \xi$.  We can recursively add points to $X_0$ until we have a good condition $p'_1 \supseteq p_1$.  To see this, take an increasing enumeration $\langle \alpha_i : i < \ot(X_0) \rangle$.  Suppose we have $p'_1 \restriction (\alpha_j \times 1)$ for $j < i$.  Let $p''_1 \restriction (\alpha_i \times 1)$ be the union of these.  If there is a horizontal extension $r$ of $\pi(p''_1 \restriction (\alpha_i \times 1))$ that decides whether $p_1(\alpha_i) \in \dot{\mathbb Q}_{\alpha_i}$, let $p'_1 \restriction (\alpha_i \times 1)$ be such that $f(p'_1(\alpha,0)) = r(\alpha)$ for all $\alpha \in \dom r \setminus X$.  We take $p'_1$ to be the union of these over $i <\ot(X_0)$.  Let $\dom p'_1 = X_1 \times 1$.

Next we extend $p_2$ to have domain $X_1 \times 2$, by putting $p''_2(\alpha,1) = p'_1(\alpha,0)$ for $\alpha \in X_1 \setminus X_0$.  Then we apply the same procedure to add points to $X_1$ to produce $X_2$ and $p'_2$, this time also taking care to define $p'_2(\alpha,0) = 0$ when a new point $\alpha$ is introduced.  This will ensure $p'_2$ is flat, and also not alter the fact that $p'_2 \restriction X_2 \times 1$ is good.  We continue in this way up to $\xi$, adding zeros to the lower positions when necessary and simply taking unions at limits.  
The resulting condition will extend $p$, be flat with domain $X_\xi \times \xi$, and be good.

Suppose $q \leq p$ are good flat conditions, where $\dom p = X \times \zeta$, and $\dom q = Y \times \eta$.  
We show by induction on $i < \zeta$, and in each case by induction on $\alpha < \lambda$, that $\pi(q_i) \restriction \alpha$ is a horizontal extension of $\pi(p_i) \restriction \alpha$.  Suppose this is true for $j < i$ and $\beta < \alpha$.  If $i$ is a limit, the induction proceeds trivially.  Suppose $i$ is a successor and $\alpha \in Y \setminus X$.  Then $\pi(p_i)$ is trivial at $\alpha$, so $\pi(q_i) \restriction (\alpha+1)$ is a horizontal extension.  Otherwise, we have $\pi(p_i) \restriction \alpha$ and $\pi(q_i) \restriction \alpha$ must either both not decide, or both agree whether the set $f(p(\alpha,i-1)) = f(q(\alpha,i-1))$ is in $\dot{\mathbb Q}_\alpha$.  Since by induction $\pi(p_{i-1})(\alpha) = \pi(q_{i-1})(\alpha)$, we must have $\pi(p_i)(\alpha) = \pi(q_i)(\alpha)$, showing that we may continue the induction along $\lambda$ in the case of successor $i$.  We conclude in particular that $\pi(q) \leq \pi(p)$.

Now suppose $p \in \add(\kappa,\lambda)$, $\dom p = X \times \xi$, and $q \leq \pi(p)$ is in $\overline{\mathbb P_\lambda}$.  For each $\alpha \in \sprt(q)$, there is $r_\alpha \in [\kappa]^{<\kappa}$ such that $q {\restriction} \alpha \Vdash_\alpha q(\alpha) = \check r_\alpha$.  If $\alpha \in X$, define $p'(\alpha,\xi) = f^{-1}(r_\alpha)$.  If $\alpha \in \dom q \setminus X$, let $p'(\alpha,i) = 0$ for all $i < \xi$, and $p'(\alpha,\xi) = f^{-1}(r_\alpha)$.  Since $\pi(p' \restriction (\sprt(q) \times \xi)) = \pi(p)$, we have $\pi(p') = q$.
\end{proof}

\begin{remark}
\label{niceproj}
Suppose we are in the situation as in the previous lemma.  Suppose $C \times G_\lambda$ is $(\mathbb C(\emptyset) \times \mathbb P_\lambda)$-generic, and let $G_\alpha$ be the generic for the subforcing $\mathbb P_\alpha$ for $\alpha < \lambda$. 
Then for all $p \in \add(\kappa,\lambda)$, $\pi_\lambda(p) \in G_\lambda$ iff $\pi_\alpha(p \restriction \alpha) \in G_\alpha$ for all $\alpha < \lambda$.  It follows that for all $\alpha < \lambda$, the identity map is a regular embedding from the quotient $\add(\kappa,\alpha)/G_\alpha$ into $\add(\kappa,\lambda)/G_\lambda$.  Therefore, if $\lambda$ is regular and $\alpha^{<\kappa} < \lambda$ for all $\alpha < \lambda$, then $\add(\kappa,\lambda)/G_\lambda$ is $(\lambda \cap \cof(\kappa))$-layered.
\end{remark}

\begin{theorem}
\label{localize}
Suppose the following:
\begin{enumerate}
\item GCH, $\mu$ is regular, and $\kappa = \mu^+$.
\item $I$ is a normal $\kappa^+$-saturated ideal on $\kappa$.
\item There is a stationary $A \in I^*$ and a nonstationary $B \subseteq \kappa^+$ such that 
$\Vdash_{\p(\kappa)/I} j(\check A) = \check B$.
\item There is a projection $\pi :  \p(\kappa)/I \to \col(\mu,\kappa) \times \add(\kappa,\kappa^+)$.
\end{enumerate}
Let $S = \kappa \cap \cof(\mu) \setminus A$.  Then $S$ is stationary, and there is an $S$-iteration $\mathbb P$ of length $\kappa^+$ such that \[\mathcal B(\mathbb P * \p(A) / \ns_\kappa) \cong \p(\kappa) / I.\]
\end{theorem}

\begin{proof}
If $j : V \to M \subseteq V[G]$ is any generic ultrapower arising from $I$, then $\cof(\mu)^M = \cof(\mu)^{V[G]}$.  Thus since $B$ not almost all of $\cof(\mu) \cap j(\kappa)$, $S$ is stationary.
In $V$, we will construct an $S$-iteration $\mathbb P$ of length $\kappa^+$, and simultaneously construct a sequence of projections from $\p(\kappa)/I$ to $\mathbb P$ and a sequence of $\mathbb P$-names for ideals.

Suppose $G \subseteq \p(\kappa)/I$ is generic, and $j : V \to M \subseteq V[G]$ is the generic ultrapower embedding.  Since $B$ is nonstationary, the forcing $\mathbb C(j(S) \cup \cof({<}\mu))^M$ has a $\kappa^+$-closed dense set in $V[G]$, as does $\add(j(\kappa),j(\kappa^+))^M$.  Since $\mathbb C(j(S) \cup \cof({<}\mu)) \times \add(j(\kappa),j(\kappa^+))$ has the $j(\kappa^+)$-c.c.\ in $M$, and $j(\kappa^+)< (\kappa^{++})^V$, there are only $j(\kappa)$-many dense open subsets that live in $M$.  Using this and the closure of $M$, we can build in $V[G]$ a filter $H$ that is generic over $M$.  Fix a $\p(\kappa)/I$-name for such an object.

In $V$, let $\langle S^0_\gamma : \gamma < \kappa^+ \rangle$ enumerate $\p(\kappa)$.  We begin our $S$-iteration by first forcing with $\mathbb C(S^0_0)$ if $S^0_0 \in I^*$, and forcing with $\mathbb C(\kappa)$ otherwise.  Since $\col(\mu,\kappa)$ absorbs $\mathbb C(S \cup \cof({<}\mu))$, Lemma \ref{absorbshoot} gives a projection $\mathbb \pi_1 : \col(\mu,\kappa) \times \add(\kappa,1) \to \mathbb C(\emptyset) \times \mathbb C(S^0_0)$, so that a generic for this first step is absorbed by $\p(\kappa)/I$.  Suppose $C_0$ is a generic club for $\mathbb P_1$, and $G \subseteq \p(\kappa)/I$ is a generic absorbing $C_0$ via $\pi_1$.  Note that $C_0 \cup \{ \kappa \}$ is a condition in $j(\mathbb P_1)$.  
The generic filter $H$ yields, via the projection of Lemma \ref{absorbshoot}, a club $\hat C_0 \subseteq j(\kappa)$ that is $j(\mathbb P_1) \restriction C_0 \cup \{\kappa\}$-generic over $M$.
Thus we can extend the embedding to $j_1 : V[C_0] \to M[\hat C_0]$.  By Theorem \ref{dualitygen}, we have in $V$ a $\mathbb P_1$-name $\dot I_1$ for a normal ideal extending $I$ such that $\mathcal B(\mathbb P_1 * \p(\kappa) / \dot I_1) \cong \p(\kappa)/I$, where $I_1$ is the collection of $X \subseteq \kappa$ for which it is forced that $\kappa \notin j_1(X)$.

Let $\alpha \mapsto \langle \alpha_0,\alpha_1 \rangle$ denote the G\"odel pairing function; note that $\alpha_0,\alpha_1 \leq \alpha$.  Let $\alpha < \kappa^+$ and assume inductively that:
\begin{enumerate}
\item In $V$, we have an $S$-iteration $\la \mathbb P_\beta,\dot{\mathbb Q}_\beta : \beta < \alpha \ra$.
\item  At each stage $\beta < \alpha$, we have chosen an enumeration $\langle \dot S^\beta_\gamma : \gamma < \kappa^+ \rangle$ of $\p(\kappa)^{V^{\mathbb P_\beta}}$.
\item We have a commuting system of projections
$\pi_\beta : \col(\mu,\kappa) \times \add(\kappa,\beta) \to \mathbb P_\beta$, for $\beta \leq \alpha$, as given by Lemma \ref{absorbshoot}.  This implies that $\p(\kappa)/I$ absorbs $\mathbb P_\alpha$.  Let $C_\beta$ denote the generic club added at stage $\beta$.
\item For $\beta < \alpha$, it is forced that $m_\beta = \la (j(\gamma), C_\gamma \cup \{ \kappa \}) : \gamma < \beta \ra$ is a condition in $\overline{j(\mathbb P_\beta)}$.
\item For $\beta<\alpha$, $H$ yields a $j(\mathbb P_\beta) \restriction m_\beta$-generic sequence $\la \hat C_\gamma : \gamma < j(\beta) \ra$.  By the coherence of the projections from Lemma \ref{absorbshoot}, we have that for $\beta'<\beta$, the sequence associated to $\beta$ is an end-extension of that associated to $\beta'$.
\end{enumerate}

Since we use ${<}\kappa$-supports, for all $\beta <\alpha$ and all $p$ in the generic filter corresponding to $\la C_\gamma : \gamma < \beta \ra$, $m_\beta \leq j(p)$.  Thus for all $\beta < \alpha$, we have elementary embeddings $j_\beta : V[\la C_\gamma : \gamma < \beta \ra] \to M[\la \hat C_\gamma : \gamma < j(\beta) \ra] \subseteq V[G]$, which are forced to extend one another.  Theorem \ref{dualitygen} gives that for each $\beta < \alpha$, there is a $\mathbb P_\beta$-name $\dot I_\beta$ for a normal ideal on $\kappa$ equal to the set of $X \subseteq \kappa$ for which it is forced that $\kappa \notin j_\beta(X)$, and $\mathcal B(\mathbb P_\beta * \p(\kappa) / \dot I_\beta) \cong \p(\kappa)/I$.

If $\alpha = \beta+1$, then we continue the construction by letting $\dot{\mathbb Q}_\beta$ be a $\mathbb P_\beta$-name for $\mathbb C(\dot S^{\beta_0}_{\beta_1})$ if $\Vdash_\beta \dot S^{\beta_0}_{\beta_1} \in \dot I_\beta^*$ and $\mathbb C(\kappa)$ otherwise.  We then choose an enumeration of $\p(\kappa)^{V^{\mathbb P_\alpha}}$.  The first three induction hypotheses are easily preserved.  For (4), first note that it is preserved at successors because $\overline{\mathbb P_\alpha}$ is dense in $\mathbb P_\alpha$, and because $\kappa$ is forced to be in $j(S^{\beta_0}_{\beta_1})$.   At limits, it is preserved by the closure of $M$ and because $j(\mathbb P_\alpha)$ uses ${<}j(\kappa)$-supports.  Thus (5) makes sense and follows from Lemma \ref{absorbshoot}.

Regarding the final stage $\kappa^+$, note that if $\mathcal A \subseteq j(\mathbb P_{\kappa^+})$ is a maximal antichain in $M$, then $\mathcal A \subseteq j(\mathbb P_{\beta})$ for some $\beta < \kappa^+$.  Thus the filter induced by $\la \hat C_\gamma : \gamma < \beta \ra$ meets $\mathcal A$.  Thus $\la \hat C_\gamma : \gamma < j(\kappa^+) \ra$ is $j(\mathbb P_{\kappa^+})$-generic over $M$, and we may extend the embedding to $j_{\kappa^+} : V[\la C_\gamma : \gamma < \kappa^+ \ra] \to M[\la \hat C_\gamma : \gamma < j(\kappa^+) \ra] \subseteq V[G]$.


Suppose $\dot X$ is a $\mathbb P_{\kappa^+}$-name for a set in $I_{\kappa^+}^*$.  Then $\dot X$ is a $\mathbb P_\beta$-name for some $\beta < \kappa^+$ by the $\kappa^+$-c.c.  If $\nVdash_{\mathbb P_\beta} \dot X \in I_\beta^*$, then we can take a generic $G \subseteq \p(\kappa)/I$ such that $\kappa \notin j_\beta(X)$.  But $G$ also determines an embedding $j_{\kappa^+}$ extending $j_\beta$, and by assumption it is forced that $\kappa \in j_{\kappa^+}(X)$, a contradiction.  Now in $V^{\mathbb P_\beta}$, $X = S^\beta_\gamma$ for some $\gamma$, and there is $\alpha \geq \beta$ such that $\langle \alpha_0,\alpha_1 \rangle = \langle \beta,\gamma \rangle$.  Thus $\mathbb Q_\alpha$ shoots a club $C$ through $X \cup (\kappa \setminus A)$, so that in the final model, $C \cap A \subseteq X$.
\end{proof}

\section{Obtaining saturated ideals}
\label{buildingblock}

Let $\mu<\kappa$ be regular cardinals.  Deviating slightly from convention, we define:
$$\col(\mu,{<}\kappa) :=   \prod^{{<}\mu \text{-support}}_{\substack{\mu \leq \alpha<\kappa \\ \alpha \text{ regular}}} \col(\mu,\alpha).$$
Now define: 
$$\mathbb P(\mu,\kappa) := \prod^{\text{Easton support}}_{\substack{\mu \leq \alpha<\kappa \\ \alpha \text{ regular}}} \col(\alpha,{<}\kappa).$$
It is convenient to identity $\mathbb P(\mu,\kappa)$ with a collection of partial functions on $\kappa^3$.
If $\mu<\delta < \kappa$ and $\delta$ is regular, then $\mathbb P(\mu,\delta) = \mathbb P(\mu,\kappa) \cap V_\delta$.  Thus if $\kappa$ is Mahlo, then $\mathbb P(\mu,\kappa)$ is $S$-layered, where $S$ is the stationary set of regular cardinals below $\kappa$.

The key to our construction is Shioya's argument \cite{shioyaeaston}, which shows that $\mathbb P(\mu,\kappa)$ can absorb future versions of itself.

\begin{lemma}
\label{newproj}
Assume GCH.  Suppose $\kappa < \lambda$ are regular cardinals, and $\mathbb Q$ is a $\kappa$-c.c.\ partial order of size $\leq \kappa$.  There is a projection $\pi : \mathbb Q \times \mathbb P(\kappa,\lambda) \to \mathbb Q * \dot{\mathbb P}(\kappa,\lambda)$.
\end{lemma}

\begin{proof}
Using GCH and the chain condition, choose an enumeration $\{ \tau_\alpha : \alpha < \lambda \}$ of $\mathbb Q$-names for ordinals such that whenever $\eta \in [\kappa, \lambda]$ is regular and $\Vdash \sigma < \check \eta$, there is $\alpha < \eta$ such that $\Vdash \sigma = \tau_\alpha$.  Let $\pi$ be defined by $\la q,p \ra \mapsto \langle q, \dot p\rangle$, where $\dot p$ is the canonical $\mathbb Q$-name for the function with the same domain as $p$, and such that for all $\vec\alpha \in \dom p$, $\Vdash \dot p(\vec \alpha) = \tau_{p(\vec \alpha)}$.

Suppose $\langle q_1, \dot p_1 \rangle \leq \pi(q_0,p_0)$.  By the $\kappa$-c.c., there is a set $X \subseteq \lambda^3$ such that $\Vdash \dom \dot p_1 \subseteq \check X$, where:
\begin{itemize}
\item $\{ \alpha : \exists \beta \exists \gamma\langle \alpha,\beta,\gamma \rangle \in X \} \subseteq [\kappa,\lambda)$ and is Easton.
\item $\forall \alpha \{ \la \beta,\gamma \ra :\langle \alpha,\beta,\gamma \rangle \in X \} \subseteq [\alpha,\lambda) \times \alpha$ and has size $<\alpha$.
\end{itemize}
Define $p_2$ such that $p_2 \restriction \dom p_0 = p_0$, and if $\vec \alpha \in  X  \setminus \dom p_0$, then $p_2(\vec \alpha) = \delta$, where $\delta<\vec \alpha(1)$ is such that the following is forced about the $\mathbb Q$-name $\tau_\delta$:
``If $\vec \alpha \in \dom \dot{p}_1$, then $\tau_\delta = \dot{p}_1( \vec \alpha)$.''  If $\vec \alpha \in \dom p_0$, then $q_1 \Vdash \dot{p}_0(\vec\alpha) =  \dot{p}_1(\vec \alpha) =  \dot{p}_2(\vec\alpha)$.  If $\vec \alpha \in \dom p_2 \setminus \dom p_0$, then $q_1 \Vdash \vec \alpha \in \dom \dot{p}_1 \rightarrow \dot{p}_1(\vec\alpha) = \dot{p}_2(\vec\alpha)$.  Thus $q_1 \Vdash \dot{p}_2 \leq \dot{p}_1$.
\end{proof}

\begin{lemma}
\label{ahiso}
Suppose the following:
\begin{enumerate}
\item $\mu<\kappa\leq\delta<\lambda$ are regular cardinals, and $\kappa$ and $\lambda$ are Mahlo.
\item $j : V \to M$ is an almost-huge embedding derived from a $(\kappa,\lambda)$-tower.
\item $\dot{\mathbb Q}$ is a $\mathbb{P}(\mu,\kappa)$-name for a $\kappa$-closed, $\delta$-c.c.\ poset of size $\leq \delta$.
\end{enumerate}
Let $G * h * H$ be $\mathbb P(\mu,\kappa) * \dot{\mathbb Q} * \dot{\mathbb P}(\delta,\lambda)$-generic.  In $V[G*h*H]$, there is a normal $\kappa$-complete ideal $I$ on $[\delta]^{<\kappa}$ such that $\p([\delta]^{<\kappa})/I$ projects to $\col(\mu,\kappa) \times \col(\kappa,\delta) \times \add(\kappa,\lambda)$, and is $S$-layered, where $S$ is the set of $V$-regular $\alpha$, $\delta \leq \alpha < \lambda$.  Any generic embedding arising from forcing with this ideal extends $j$.
\end{lemma}

\begin{proof}
Let us first claim that there is a projection from $\mathbb P(\mu,\lambda)$ to 
$$\mathbb P(\mu,\kappa) * [(\dot{\mathbb Q} * \dot{\mathbb P}(\delta,\lambda)) \times \dot\col(\kappa,\delta) \times \dot\add(\kappa,\lambda)] \times \col(\mu,\kappa),$$ which is the identity on $\mathbb P(\mu,\kappa)$.  Note that there is a natural projection from $\mathbb P(\mu,\lambda)$ to $\mathbb P(\mu,\kappa) \times \mathbb P(\kappa,\lambda) \times  \col(\mu,\kappa)$, given by 
$$p \mapsto \langle p \restriction \kappa^3, p \restriction [\kappa,\lambda) \times \lambda^2,p(\mu,\kappa,\cdot) \rangle.$$
By Lemma \ref{newproj}, the first two factors project to $\mathbb  P(\mu,\kappa) * \dot{\mathbb P}(\kappa,\lambda)$.  Since by Lemma \ref{folk}, $\col(\alpha,\beta) \cong \col(\alpha,\beta) \times \mathbb R$, whenever $\alpha$ is regular and $\mathbb R$ is $\alpha$-closed and of size $\leq \beta$, we see that in $V^{\mathbb P(\mu,\kappa)}$, $\mathbb P(\kappa,\lambda)$ is forcing-equivalent to $\mathbb P(\kappa,\lambda) \times \mathbb P(\delta,\lambda) \times \col(\kappa,\delta) \times \add(\kappa,\lambda)$.  The first factor projects to $\mathbb Q$ by Lemma \ref{folk} again.  By Lemma \ref{newproj}, we get a projection from the first two factors to $\mathbb Q * \dot{\mathbb P}(\delta,\lambda)$.  This finishes the argument for the claim.

Now let $G * h * H$ be as hypothesized, and let $\hat G \subseteq \mathbb P(\mu,\lambda)$ be a generic projecting to $G * h * H$.  By the above argument, $\hat G$ also projects to $G * K \subseteq \mathbb P(\mu,\kappa)* \dot{\mathbb P}(\kappa,\lambda)$, which in turn projects to $G * h * H$.  Since $\mathbb P(\kappa,\lambda)$ is clearly $(\kappa,\lambda)$-nice, we get by Lemma \ref{ahext} an extended elementary embedding $j : V[G * K] \to M[\hat G * \hat K]$, such that $\ord^{<\lambda} \cap V[\hat G] \subseteq M[\hat G * \hat K]$.  By elementarity, there is some $\hat h * \hat H$ such that the embedding restricts to $j : V[G * h * H] \to M[\hat G * \hat h * \hat H]$.  By the closure of the relevant forcings, the latter model also has the same $\ord^{<\lambda}$ as $V[\hat G]$.

 In $V[G*h*H]$, let $\mathbb R$ be the quotient forcing $\mathbb P(\mu,\lambda)/(G*h*H)$.  We define the ideal $I$ as $\{ X \subseteq [\delta]^{<\kappa} : 1 \Vdash_{\mathbb R} j[\delta] \notin j(X) \}$.  Let $e : \p([\delta]^{<\kappa})/I \to \mathcal{B}(\mathbb R)$ be defined by $e([X]_I) = || j[\delta] \in j(X) ||$.  It is routine to check that $I$ is normal and $\kappa$-complete and that $e$ is a Boolean embedding.  Since $\mathbb P(\mu,\lambda)$ is $\lambda$-c.c., $I$ is $\lambda$-saturated.  It follows from this that $e$ is a complete embedding.  For let $\{ [X_\alpha]_I : \alpha < \delta \}$ be a maximal antichain in $\p([\delta]^{<\kappa})/I$.  Then $[\nabla_{\alpha<\delta} X_\alpha]_I = [[\delta]^{<\kappa}]_I$, so $1 \Vdash_{\mathbb R}  j[\delta] \in j(\nabla_{\alpha<\delta} X_\alpha)$.  Thus it is forced that $j[\delta] \in j(X_\alpha)$ for some $\alpha < \delta$.

To show the isomorphism, let $U \subseteq \p([\delta]^{<\kappa}) / I$ be generic over $V[G*h*H]$, and let $j_U : V[G*h*H] \to N$ be the generic ultrapower embedding.  Forcing further with $\mathbb R/e[U]$ yields an embedding $j : V[G*h*H] \to M[\hat G * \hat h * \hat H]$ as above.  We have that $X \in U$ if and only if $j[\delta] \in j(X)$, so we can define an elementary embedding $k : N \to M[\hat G * \hat h * \hat H]$ by $k([f]_U) = j(f)(j[\delta])$, and we have $j = k \circ j_U$.  Note that for $\alpha \leq \delta$, $k(\alpha) = k(\ot(j_U(\alpha) \cap [\id]_U)) = \ot(j(\alpha) \cap j[\delta]) = \alpha$.  Thus $\crit(k) \geq \lambda$.


Let $\beta$ be any ordinal.  Since $j: V \to M$ was derived from a $(\kappa,\lambda)$-tower, there is some $\alpha < \delta$ and some $f \in V$ such that $\beta = j(f)(j[\alpha])$.  Let $b : \delta \to \alpha$ be a surjection in $V[G*h*H]$.  Then 
$$\beta = j(f)(j(b)[j[\delta]]) = k( j_U(f)(j_U(b)[[\id]_U])).$$
  Thus $\beta \in \ran(k)$, so $k$ does not have a critical point and $N = M[\hat{G} * \hat h * \hat{H}]$.
For any generic $\hat{G}$, if $U = e^{-1}[\hat G]$, then $\hat G =  j_U(G)$.  For any generic $U$, if $\hat G = j_U(G)$, then $U = \{ X \subseteq [\delta]^{<\kappa} : j[\delta] \in j_U(X) \} = e^{-1}[\hat G]$.  Thus Lemma~\ref{forcingiso} implies that $\p([\delta]^{<\kappa}) / I \cong \mathcal{B}(\mathbb R)$.
It follows that $\p([\delta]^{<\kappa}) / I$ projects to $\col(\mu,\kappa) \times \col(\kappa,\delta) \times \add(\kappa,\lambda)$ as desired.

It remains to show that $\p([\delta]^{<\kappa}) / I$ is $S$-layered, where $S$ is the stationary set of $V$-regular cardinals between $\delta$ and $\lambda$.  This follows because the projection $\pi$ of Lemma \ref{newproj} has the following property: For any $p$ in the quotient $\mathbb R$, any $\alpha < \lambda$, and any $q \leq p \restriction \alpha$ of rank $\leq \alpha$ which is also in the quotient, we have that $p \cup q$ is also in the quotient.  Thus $\la \mathbb R \cap V_\alpha : \alpha < \lambda \ra$ witnesses that $\mathbb R$ is $S$-layered.
\end{proof}

Although we are ultimately interested in saturated ideals on regular cardinals, rather than sets of the form $[\delta]^{<\kappa}$, the ideals on such sets will be useful for us because of their resilience under collapses:

\begin{lemma}
\label{collapsedelta}
Assume GCH.  Suppose $\kappa<\delta$ are regular, and $I$ is a normal ideal on $[\delta]^{<\kappa}$ such that $\p([\delta]^{<\kappa})/I$ projects to $\col(\kappa,\delta)$, and is $S$-layered for some stationary subset $S \subseteq \delta^+$.  If $g \subseteq \col(\kappa,\delta)$ is generic, then in $V[G]$, there is an $S$-layered ideal $J$ on $\kappa$ such that $\p(\kappa)/J \cong (\p([\delta]^{<\kappa})/I)/g$.  Furthermore, any generic ultrapower arising from forcing with $J$ over $V[g]$ extends one arising from forcing with $I$ over $V$.
\end{lemma}

\begin{proof}
Let $g \subseteq \col(\kappa,\delta)$ be generic.  Further forcing yields a generic $G \subseteq \p([\delta]^{<\kappa})/I$ and an ultrapower embedding $j : V \to M \subseteq V[G]$.  The quotient $(\p([\delta]^{<\kappa})/I)/g$ is $S$-layered by Lemma \ref{iterlayer}. Since $j(\col(\kappa,\delta))$ is $j(\kappa)$-directed-closed and $j[g] \in M$ is a directed set of size $<j(\kappa)$, there is a condition $m \in j(\col(\kappa,\delta))$ below $j[g]$.

A counting argument shows that $j(\delta^+) < \delta^{++}$, so there are only $\delta^+$-many dense subsets of $j(\col(\kappa,\delta))$ in $M$.  Since $j(\kappa) > \delta$ and $M^\delta \cap V[G] \subseteq M$, we can build a filter $\hat g \subseteq  j(\col(\kappa,\delta))$ in $V[G]$ that is generic over $M$, with $m \in \hat g$.  Thus we may extend the embedding to $j : V[g] \to M[\hat g]$.  By Theorem \ref{dualitygen}, we get a normal $\kappa$-complete ideal $J'$ on $[\delta]^{<\kappa}$ in $V[g]$ such that $\p([\delta]^{<\kappa})/J' \cong  (\p([\delta]^{<\kappa})/I)/g$.  By Proposition \ref{ultequal}, any generic embedding coming from $J'$ extends one coming from $I$.

Now since $|\delta| = \kappa$ in $V[g]$, a bijection $f : \kappa \to \delta$ yields an ideal $J$ on $[\kappa]^{<\kappa}$ given by $J = \{ X : \{ f[z] : z \in X \} \in J' \}$, and clearly $\p([\kappa]^{<\kappa}) / J \cong \p([\delta]^{<\kappa}) / J'$.  By normality, $\kappa$ is a $J$-measure-one set, so $J$ is essentially an ideal on $\kappa$.
\end{proof}

\subsection{Interlude: Woodin's theorem}
Here we discuss how our techniques yield a new proof of Woodin's theorem that if $\kappa$ is almost-huge and $\mu<\kappa$ is regular, then there is a forcing that preserves cardinals $\leq \mu$ and forces that $\kappa = \mu^+$ and the nonstationary ideal on $\kappa$ is locally saturated.

Let $\lambda$ be least such that there exists a $(\kappa,\lambda)$-tower $T$.  It is easy to see that $\lambda$ is inaccessible but not Mahlo.  Therefore, if $A = \{ \alpha < \kappa : \alpha$ is inaccessible$\}$, and $j : V \to M$ is the embedding derived from $T$, then $\kappa \in j(A)$ and $j(A)$ is nonstationary.
Typical Easton-support products up to $\lambda$ will not be $\lambda$-c.c., so we will have to slightly modify our forcing.  For regular $\alpha<\beta$, let
$$\mathbb Q(\alpha,\beta) = \prod^{<\alpha\text{-support}}_{\substack{\alpha \leq \gamma<\beta \\ \gamma \text{ regular}}} \col(\gamma,{<}\beta).$$
The usual $\Delta$-system argument shows that $\mathbb Q(\alpha,\beta)$ is $\beta$-c.c.\ when $\beta$ is inaccessible.  The argument for Lemma \ref{newproj} shows that $j(\mathbb Q(\mu,\kappa)) = \mathbb Q(\mu,\lambda)$ projects to $\mathbb Q(\mu,\kappa) * [\dot{\col}(\kappa,{<}\lambda) \times \dot{\add}(\kappa,\lambda)] \times \col(\mu,\kappa)$.  The argument for Lemma \ref{ahiso} shows that if $G * H \subseteq \mathbb Q(\mu,\kappa) * \dot{\col}(\kappa,{<}\lambda)$ is generic, then in $V[G*H]$, there is a normal saturated ideal $I$ on $\kappa$ such that:
\begin{enumerate}
\item $A \in I^*$.
\item Any generic embedding arising from $I$ will extend $j$. 
\item $\p(\kappa)/I$ projects to $\col(\mu,\kappa) \times \add(\kappa,\kappa^+)$.
\end{enumerate}
Thus Theorem \ref{localize} implies that there is a cardinal-preserving forcing extension in which $\ns_\kappa \restriction A$ is $\kappa^+$-saturated.

\section{The preparatory model}
\label{preparation}

We build the preparatory model towards Theorem \ref{global} in three rounds.  We warn the reader that we will continually change the reference of ``$V$'' to mean whatever ground model on which we are currently focused.  Let $\theta$ be a huge cardinal.  A standard reflection argument shows that there is a large set $X \subseteq \theta$ such that for all $\alpha < \beta$ in $X$, $\alpha$ is almost-huge with target $\beta$.  In the first round, we arrange that for all such $\alpha<\beta$, there is an $A \subseteq \alpha$ and an $(\alpha,\beta)$-tower $T$ such that $\kappa \in j_T(A)$ and $j_T(A)$ is nonstationary.  We then collapse many cardinals so that $\theta$ is still very large, the set of cardinals below $\theta$ is almost equal to $X$, and there are many saturated ideals with the properties that make Theorem \ref{localize} applicable.  In the second round, we introduce square at every cardinal while preserving many superstrong cardinals and the desired saturated ideals.  In the third round, we arrange local saturation on the first few successors of Mahlo cardinals, while still preserving many superstrongs.  
Regarding the structure of the proof:
\begin{itemize}
\item We need to force squares before local saturation, because a standard argument shows that the forcing we use for $\square_\kappa$ also forces $\diamondsuit_{\kappa^+}(S)$ for any given stationary $S \subseteq \kappa^+$, thus ruining local saturation at $\kappa^+$ (see \cite[Lemma 3.11]{MR3724382}).  However, the reader who is interested only in local saturation and not in the square principles holding simultaneously may opt to simply skip Section \ref{squaresection}, as the third round of forcing may be carried out on the basis of the first round's preparations.  
\item The forcing of the third round is a product, but it could also be done as an iteration.  The latter may be interesting if one wants to preserve larger cardinals while getting local saturation at all successors of regulars, but the former harmonizes more with the techniques of the second round and the Radin forcing of Section \ref{radinsection}, which seem not to have such flexibility.
\end{itemize}

\subsection{Many saturated ideals that get stationarity wrong}
As described above, Woodin obtained 
local saturation at a single successor cardinal by exploiting a precise degree of almost-hugeness: A $(\kappa,\lambda)$-tower was chosen with $\lambda$ non-Mahlo.  Such towers 
will not serve our purposes here since we want the target to also be almost-huge.  In order to get almost-huge towers of very large height that also map some large subset of the critical point to a nonstationary set, we use a forcing argument supplied by Toshimichi Usuba.

\begin{lemma}[Usuba]
There is a forcing $\mathbb P$ such that whenever $\kappa$ is almost-huge in $V$ with Mahlo target $\lambda$, then in $V^{\mathbb P}$, there is a $(\kappa,\lambda)$-tower $T$ and an $A \subseteq \kappa$ such that $\kappa \in j_T(A)$ and $j_T(A)$ is nonstationary.
\end{lemma}

\begin{proof}
Let $\mathbb P$ be the Easton-support iteration of adding a Cohen subset of $\alpha$ whenever $\alpha$ is inaccessible.  Let $j : V \to M$ be an almost-huge embedding generated by a $(\kappa,\lambda)$-tower in $V$, where $\lambda$ is Mahlo.  Let $G$ be generic for $\mathbb P$, and let $G_\alpha = G \cap \mathbb P_\alpha$.  Then $j$ can be extended to $j : V[G_\kappa] \to M[G_\lambda]$.

It is easy to show that the Cohen-generic function $g : \kappa \to 2$ added at stage $\kappa$ has the property that if $A = \{ \alpha : g(\alpha) = 1 \}$ and $S \in V$ is a stationary subset of $\kappa$, then $S \cap A$ and $S \setminus A$ are both stationary.  We now build a subset of $\lambda$ that is $\add(\lambda)$-generic over $M[G_\lambda]$ with some specific properties.  Since $j(\lambda) < \lambda^+$, we can list all dense open subsets of $\add(\lambda)^{M[G_\lambda]}$ that live in $M[G_\lambda]$ as $\langle D_\alpha : \alpha < \lambda \rangle$.  We construct an extension $\hat g$ of $g$ as $\bigcup_{\alpha <\lambda} \hat g_\alpha$ and along the way choose a continuous, increasing, cofinal sequence of ordinals $\langle \beta_\alpha : \alpha < \lambda \rangle \subseteq \lambda$ with the following properties:
\begin{enumerate}
\item $\dom(\hat g_0) = \kappa +1$, $\hat g_0 \restriction \kappa = g$ and $\hat g_0(\kappa) = 1$.
\item For $\alpha > 0$, $\dom(\hat g_\alpha) = \beta_\alpha +1$, and $\hat g_\alpha(\beta_\alpha) = 0$.
\item For all $\alpha$, $\hat g_{\alpha+1} \in D_\alpha$.
\item For limit $\alpha$, $\beta_\alpha = \sup_{\gamma < \alpha} \beta_\gamma$.
\end{enumerate}

Clearly, $\hat g$ is generic over $M[G_\lambda]$, and $\{ \alpha : \hat g(\alpha) = 1 \}$ is disjoint from the club $\{ \beta_\alpha : \alpha < \lambda \}$.  So we extend the embedding to $j : V[G_{\kappa+1}] \to M[G_\lambda * \hat g]$.

The method of the proof of Theorem~\ref{ahext} lets us build a filter $H \subseteq j(\mathbb P_{\lambda}) / (G_\lambda * \hat g)$-generic over $M[G_\lambda * \hat g]$, with $j[G_\lambda] \subseteq G_\lambda * \hat g * H$, so we can extend the embedding to $j : V[G_\lambda] \to M[G_\lambda * \hat g * H]$. By the $\lambda$-c.c.\ of $\mathbb P_\lambda$ and the $\lambda$-closure of $j(\mathbb P_{\lambda}) / G_\lambda$, we have that $\ord^{<\lambda} \cap V[G_\lambda] \subseteq M[G_\lambda * \hat g * H]$.  The appropriate $(\kappa,\lambda)$-tower inducing $j$ exists in $V[G_\lambda]$, and it is preserved by the $\lambda$-closed forcing $\mathbb P / G_\lambda$. 
\end{proof}

It is not hard to show that the forcing $\mathbb P$ of the previous lemma preserves huge cardinals as well.  Let us therefore work in a model satisfying the conclusion of the previous lemma, and in which there is a huge cardinal $\theta$.  Let $\mathcal U$ be an ultrafilter on $\theta$ derived from an embedding witnessing $\theta$ is huge.  $X \in \mathcal U$ be such that for $\alpha <\beta$ in $X$, $\alpha$ is almost-huge with target $\beta$.  Let $\la \alpha_i : i < \theta \ra$ enumerate the closure of $X \cup \{ \omega \}$.  Let $\mathbb P(\kappa,\lambda)$ denote the product forcing defined in Section \ref{buildingblock}.  Let us force with the following Easton-support iteration $\la \mathbb P_i,\dot{\mathbb Q}_i : i < \theta \ra$:
\begin{itemize}
\item If $i$ is 0 or a successor, let $\Vdash_i \dot{\mathbb Q}_i = \dot{\mathbb P}(\alpha_i,\alpha_{i+1})$.
\item If $\alpha_i$ is a non-Mahlo limit of $X$, let $\Vdash_i \dot{\mathbb Q}_i = \dot{\mathbb P}(\alpha_i^+,\alpha_{i+1})$.
\item If $\alpha_i$ is a Mahlo limit of $X$, let $\Vdash_i \dot{\mathbb Q}_i = \dot{\mathbb P}(\alpha_i,\alpha_{i+1})$.
\end{itemize}
It is routine to check that after forcing with this iteration, the set of cardinals below $\theta$ are the ordinals $\alpha_i$ and those of the form $(\alpha_i^+)^V$ for $\alpha_i$ a non-Mahlo limit of $X$.

Let $\mu < \delta$ be either successor cardinals or a Mahlo cardinals after forcing with $\mathbb P_\theta$.  Let $i<\theta$ be such that either $\mu = \alpha_i$ or $\mu = (\alpha_i^+)^V$, and define $i'$ similarly with respect to $\delta$.  Let $G_i \subseteq \mathbb P_i$ be generic, and let $\kappa = \alpha_{i+1}$.  Since $|\mathbb P_i| \leq \mu$, any almost-huge embedding with critical point $\kappa$ in $V$ extends to one in $V[G_i]$.  Consider the forcing $\mathbb P_{i'+1}/G_i$.  It takes the form $\mathbb P(\mu,\kappa) * \dot{\mathbb Q} * \dot{\mathbb P}(\delta,\lambda)$, where $\lambda = \alpha_{i'+1} \in X$ and $\dot{\mathbb Q}$ is forced to be $\kappa$-closed, $\delta$-c.c., and of size $\leq \delta$.  The hypotheses of Lemma \ref{ahiso} are satisfied, so there exists an ideal as in the conclusion after forcing with $\mathbb P_{i'+1}$.  As the tail-end is $\lambda$-closed, this still holds in the extension by $\mathbb P_\theta$.

Furthermore, many almost-huge cardinals below $\theta$ are preserved.  For let $j : V \to M$ witness the hugeness of $\theta$.  If $T$ is the almost-huge tower derived from $j$, then $T \in M$.  We have $j(\mathbb P_\theta) \cap V_\theta = \mathbb P_\theta$, so reflection gives us that, if $\mathcal U$ is the ultrafilter on $\theta$ derived from $j$, then there are $\mathcal U$-many $\alpha<\theta$ that are almost-huge with embedding $j_\alpha$, with $j_\alpha(\mathbb P_\alpha)= \mathbb P_\theta$.  Reflecting again yields a set $Y \in \mathcal U$ such that for all $\alpha < \beta$ in $Y$, there is an $(\alpha,\beta)$-tower $T$ with $j_T(\mathbb P_\alpha) = \mathbb P_\beta$.  The proof of Lemma \ref{ahext} shows that such embeddings can be extended through the forcing.  Let us record what we have as:

\begin{lemma}
\label{firstprep}
It is consistent relative to a huge cardinal that there is an inaccessible $\theta$ and a sequence $\la S_\alpha : \alpha < \theta \ra$ such that:
\begin{enumerate} 
\item $V_\theta$ satisfies GCH and that there is a proper class of almost-huge cardinals with Mahlo targets.\
\item \label{statprops} Whenever $\mu$ is regular and $\kappa = \mu^+$, $S_\kappa$ is a stationary subset of $\kappa \cap \cof(\mu)$.
\item \label{idealprops} For every pair of cardinals $\mu<\delta<\theta$ which are either successor or Mahlo, if $\kappa = \mu^+$ and $\lambda= \delta^+$, then there is $\kappa$-complete normal ideal $I$ on $[\delta]^{<\kappa}$ such that: 
\begin{enumerate}
\item $\p([\delta]^{<\kappa})/I$ is $S_\lambda$-layered.
\item There is a stationary $A \subseteq \kappa \cap \cof(\mu)$ and a nonstationary $B \subseteq \lambda$ such that for any generic embedding $j$ arising from $I$, $\kappa \in j(A) = B$.
\item $\p([\delta]^{<\kappa})/I$ projects to $\col(\mu,\kappa) \times \col(\kappa,\delta) \times \add(\kappa,\lambda)$, in a way such that the quotient is forced to be $S_\lambda$-layered.
\end{enumerate}\end{enumerate}
\end{lemma}

\subsection{Squares}
\label{squaresection}
For a cardinal $\kappa$, $\square_\kappa$ holds if there is a sequence $\langle C_\alpha : \alpha < \kappa^+ \rangle$ such that if $\alpha$ is a limit ordinal,
\begin{enumerate}
\item $C_\alpha$ is a club subset of $\alpha$.
\item If $\beta \in \lim C_\alpha$, then $C_\beta = C_\alpha \cap \beta$.
\item $\ot C_\alpha \leq \kappa$.
\end{enumerate}
We will refer to a sequence satisfying (1) and (2) as a \emph{coherent sequence of clubs} and a sequence satisfying all three as a \emph{$\square_\kappa$-sequence}.  A weaker notion, $\square(\kappa^+)$, holds when there is a coherent sequence of clubs with the property that there is no ``thread'' $C \subseteq \kappa^+$, a club such that if $\alpha \in \lim C$, then $C_\alpha = C \cap \alpha$.

There is some tension between squares and saturated ideals.  The following two propositions show that if $\mu$ has uncountable cofinality and $\square_\mu$ holds, then there cannot be a saturated ideal on $\mu^+$ whose associated forcing is either weakly homogeneous or proper.
\begin{proposition}[Zeman]
Suppose $I$ is a normal $\kappa$-complete ideal on $Z$, $\p(Z)/I$ is weakly homogeneous, and $\p(Z)/I$ preserves that $\kappa$ has uncountable cofinality.  Then $\square(\kappa)$ fails.
\end{proposition}
\begin{proof}
Suppose $\langle C_\alpha : \alpha < \kappa \rangle$ is a coherent sequence of clubs.  Let $G \subseteq \p(Z)/I$ be generic, and let $j : V \to M$ be the associated embedding.  By \cite[Section 2.6]{MR2768692} , $M$ is well-founded up to $\kappa^+$.  $C^* = j(\vec C)(\kappa)$ is a thread of $\langle C_\alpha : \alpha < \kappa \rangle$.  Suppose $C'$ is another thread.  Then $C'' = C' \cap C^*$ is a club in $\kappa$, and whenever $\alpha \in \lim C''$, $C'' \cap \alpha = C' \cap \alpha = C^* \cap \alpha = C_\alpha$.  Thus $C' = C^*$, hence $C^*$ is definable from parameters in the ground model.  By weak homogeneity, $C^* \in V$.
\end{proof}

\begin{proposition}
\label{proper}
Suppose $I$ is a normal $\kappa$-complete ideal on $Z$, and $\p(Z)/I$ is a proper forcing.  Then every stationary subset of $\kappa \cap \cof(\omega)$ reflects.
\end{proposition}
\begin{proof}
Let $S \subseteq \kappa \cap \cof(\omega)$ be stationary.  Let $G \subseteq \p(Z)/I$ be generic and let $j : V \to M$ be the associated embedding.  Then $j(S) \cap \kappa = S$, and $S$ is still stationary in $V[G]$.  By elementarity, $S \cap \alpha$ is stationary in $\alpha$ for some $\alpha<\kappa$.
\end{proof}

We will need the following to show the preservation of squares under some cardinal collapses: 
\begin{lemma}
\label{refine}
Let $\kappa$ be a cardinal and $\zeta < \kappa^+$.  Suppose there is a coherent sequence of clubs $\langle C_\alpha : \alpha < \kappa^+ \rangle$ such that for all $\alpha$, $\ot C_\alpha \leq \zeta$.  Then $\square_\kappa$ holds.
\end{lemma}
\begin{proof}
It is easy to show by induction that for each $\xi < \kappa^+$, there is a short square sequence $\la D_\alpha : \alpha \leq \xi \ra$, i.e. a sequence satisfying all requirements for $\square_\kappa$ except that its length is $<\kappa^+$.  Fix one for $\xi = \zeta$.  For $\alpha < \kappa^+$, let $C'_\alpha = \{ \beta \in C_\alpha: \ot(C_\alpha \cap \beta) \in D_{\ot(C_\alpha)} \}$.  For each $\alpha$, $\ot(C'_\alpha) = \ot(D_{\ot(C_\alpha)}) \leq \kappa$.  Suppose $\beta$ is a limit point of $C'_\alpha$.  Then $\beta$ is a limit point of $C_\alpha$, so $C_\beta = C_\alpha \cap \beta$.  Also, $\ot(C_\beta)$ is a limit point of $D_{\ot(C_\alpha)}$, so $D_{\ot(C_\alpha)} \cap \ot(C_\beta) = D_{\ot(C_\beta)}$, and therefore $C'_\beta = C'_\alpha \cap \beta$.
\end{proof}

For a cardinal $\delta$, let $\mathbb{S}_\delta$ be the collection of bounded approximations to a $\square_\delta$ sequence.  That is, a condition is a sequence $\langle C_\alpha : \alpha \in \eta \rangle$ such that $\eta < \delta^+$ is a successor ordinal, each $C_\alpha$ is a club subset of $\alpha$ of order type $\leq \delta$, and whenever $\beta$ is a limit point of $C_\alpha$, $C_\alpha \cap \beta = C_\beta$.  An induction argument shows that conditions can be extended to arbitrary length, so the forcing introduces a $\square_\delta$-sequence.  The first and third claims of the following lemma are well-known, and the second follows from a general theorem of Ishiu and Yoshinobu \cite{MR1879973}.  We give a proof for the reader's convenience.

\begin{lemma}Let $\delta$ be a cardinal.
\label{squarestrat}
\begin{enumerate}
\item $\mathbb{S}_\delta$ is $(\delta+1)$-strategically closed.
\item If $\square_\delta$ holds, then $\mathbb S_\delta$ is $\delta^+$-strategically closed.
\item For every regular $\lambda \leq \delta$, there is a $\mathbb{S}_\delta$-name for a ``threading'' partial order $\mathbb{T}_\delta^\lambda$ that adds a club $C \subseteq (\delta^+)^V$ of order type $\lambda$ and such that whenever $\alpha$ is a limit point of $C$, $C \cap \alpha = C_\alpha$.  Furthermore, $\mathbb{S}_\delta * \dot{\mathbb{T}}_\delta^\lambda$ has a $\lambda$-closed dense subset of size $2^\delta$.
\end{enumerate}
\end{lemma}
\begin{proof}
For (1), let us pit the players \emph{Good} and \emph{Bad} against each other.  Let \emph{Bad} play any condition $p_0$.  If \emph{Bad} plays $p_\alpha$, let \emph{Good} play any condition $p_{\beta+1}$ strictly longer than $p_\beta$, where $\max(\dom p_{\beta+1})$ is a limit ordinal.  At limit stages $\lambda$, \emph{Good} plays $\bigcup_{\gamma<\lambda} p_\gamma \cup \langle \lambda, \{ \alpha : (\exists \beta < \lambda) \max(\dom p_\beta) = \alpha \} \rangle$.  The fact that \emph{Good} plays at all limit stages ensures coherence.  This strategy succeeds in producing conditions in $\mathbb S_\delta$ for $(\delta+1)$-many turns, as the order types never get too large.

For (2), assume there is a square sequence $\langle D_\alpha : \alpha < \delta^+ \rangle$.  \emph{Good} plays a similar strategy, except at limit $\lambda$, she plays $\bigcup_{\gamma<\lambda} p_\gamma \cup \langle \lambda, \{ \alpha : (\exists \beta < \lambda) \max(\dom p_\beta) = \alpha$ and $\beta \in D_\lambda \} \rangle$.  This strategy allows the game to continue for $\delta^+$-many turns.

For (3), define $\mathbb{T}_\delta^\lambda$ as the collection of bounded approximations to the desired set.  By the strategic closure of $\mathbb S_\delta$, the collection of $\langle p, \dot q \rangle \in \mathbb{S}_\delta * \dot{\mathbb{T}}_\delta^\lambda$ such that for some $x \in V$, $p \Vdash \dot q = \check x$ is dense.  If $\langle \langle p_\alpha, \check x_\alpha \rangle : \alpha < \beta < \lambda \rangle$ is a decreasing sequence of such conditions, let $x_\beta = \bigcup_{\alpha<\beta} x_\alpha$, and let $p_\beta = \bigcup_{\alpha<\beta} p_\beta \cup \langle \sup_{\alpha<\beta}(\dom p_\alpha), \check x_\beta \rangle$.  This is a condition because for all limit points $\gamma$ of $x_\beta$, $x_\beta \cap \gamma = p_\beta(\gamma)$.
\end{proof}

The next lemma will be applied to show that when $\kappa^{<\kappa} = \kappa$, forcing with $\mathbb S_\kappa$ preserves saturated ideals on $\kappa$, provided their quotient algebras are moreover $S$-layered for some stationary $S \subseteq \kappa^+$.

\begin{lemma}
\label{stratpres}
If $\kappa^{<\mu} = \kappa$, $\mu\leq\kappa$ is regular, and $\mathbb P$ is $(\kappa+1)$-strategically closed, then $\mathbb P$ preserves stationary subsets of $\kappa^+ \cap \cof(\mu)$.
\end{lemma}
\begin{proof}
Let $\sigma$ be a strategy witnessing that $\mathbb P$ is $(\kappa+1)$-strategically closed, and let $p_0 \in \mathbb P$.  Let $S \subseteq \kappa^+ \cap \cof(\mu)$ be stationary and $\dot C$ be a $\mathbb P$-name for a club. Let $\theta$ be a large regular cardinal, and let $\langle M_\alpha : \alpha < \kappa^+ \rangle$ be an increasing continuous sequence of elementary submodels of $\langle H_\theta, \in, \mathbb P, \sigma, S \rangle$, each of size $\kappa$, having transitive intersection with $\kappa^+$, and such that $M_\alpha^{<\mu} \subseteq M_\alpha$ for all successor $\alpha$.  Let $\alpha^*$ be such that $M_{\alpha^*} \cap \kappa^+ = \alpha^* \in S$.  Let 
$\langle \beta_i : i < \mu \rangle$ be an increasing sequence converging to $\alpha^*$.  Build a descending chain $\langle p_i : i < \mu \rangle \subseteq M_{\alpha^*}$ below $p_0$ such that at odd $i$, $p_i$ decides some ordinal $\geq \beta_i$ to be in $\dot C$, and even stages are chosen by following $\sigma$.  Since $M_{\alpha^*}^{<\mu} \subseteq M_{\alpha^*}$, the construction continues, and there is a condition $p_{\mu}$ below all conditions chosen.  $p_\mu \Vdash \alpha^* \in \dot C \cap \check S$.
\end{proof}

We can now perform our second round of forcing.  We simply force with the Easton-support product of $\mathbb S_\delta$, over all infinite cardinals $\delta<\theta$.  First we check that this preserves superstrong cardinals with Mahlo target.  For a set of ordinals $X$, let $\mathbb P_X$ denote the sub-product where we restrict to indices in $X$.   Let $\kappa$ be superstrong with target $\lambda$.  $j(\mathbb P_\kappa) = \mathbb P_\lambda$, and $\mathbb P_\lambda$ is $\lambda$-c.c.  Since $\mathbb P_{\lambda \setminus \kappa}$ is $(\kappa+1)$-strategically-closed and $| \mathbb P_\kappa | = \kappa$, Easton's Lemma implies that $\mathbb P_{\lambda \setminus \kappa}$ is still $\kappa^+$-distributive after forcing with $\mathbb P_\kappa$.  By Lemma \ref{ssext}, $\kappa$ is still superstrong with target $\lambda$ after forcing with $\mathbb P_\lambda$.  $\mathbb P_{\theta\setminus\lambda}$ does not add sets of rank $<\lambda$, so the superstrongness is preserved.

Now we argue that the conclusion of Lemma \ref{firstprep} still holds after forcing square everywhere below $\theta$, but with the proper class of almost-huge cardinals replaced with a proper class of superstrong cardinals.  It will be important for the argument that we force square to hold everywhere with a product rather than an iteration.

Suppose $\mu < \delta<\theta$ are successor cardinals or Mahlo, and let $\kappa = \mu^+$ and $\lambda = \delta^+$.  Let $I$ be the ideal on $[\delta]^{<\kappa}$ as in Lemma \ref{firstprep}.  The forcing $\mathbb P_{\theta \setminus \delta}$ is $(\delta+1)$-strategically closed.  Thus it preserves the stationarity of $S_\lambda$, adds no subsets of $[\delta]^{<\kappa}$, and preserves that $I$ has all the properties as in Lemma \ref{firstprep}.  If $\kappa < \delta$, consider the forcing $\mathbb P_{[\kappa,\delta)}$.  It is a coordinate-wise projection of
$$\prod^{E-\sprt}_{\kappa \leq \nu < \delta} \mathbb{S}_\nu * \dot{\mathbb{T}}_\nu^\kappa.$$
By Lemma \ref{squarestrat}, this poset is $\kappa$-closed and has size $\delta$.  Therefore, it is absorbed by $\col(\kappa,\delta)$, and thus by the Boolean algebra $\p([\delta]^{<\kappa})/I$.  If $g \subseteq \col(\kappa,\delta)$ is generic, then as in Lemma \ref{collapsedelta}, forcing with the quotient $(\p([\delta]^{<\kappa})/I)/g$ yields an embedding $j : V[g] \to M[\hat g]$.  If $h$ is the projected generic for $\mathbb P_{[\kappa,\delta)}$, then the embedding restricts to $j : V[h] \to M[\hat h]$.  By Theorem \ref{dualitygen}, there is a normal $\kappa$-complete ideal $J$ on $[\delta]^{<\kappa}$ such that $\p([\delta]^{<\kappa})/J$ is isomorphic to the quotient $(\p([\delta]^{<\kappa})/I)/h$.  This Boolean algebra is still $S_\lambda$-layered by Lemma \ref{iterlayer}.  If we make sure to use a projection that leaves a copy of $\col(\kappa,\delta)$ behind in the quotient (which can always be done as it is isomorphic to its square), we have that $\p([\delta]^{<\kappa})/J$ still projects to $\col(\mu,\kappa) \times \col(\kappa,\delta) \times \add(\kappa,\lambda)$.  Since generic embeddings via $J$ extend those via $I$, the desired property of mapping a stationary $A \subseteq \kappa$ to a nonstationary $B \subseteq \lambda$ still holds.

Now we move to the extension by $\mathbb S_\mu$, and it is here that we use the fact we have already forced with $\mathbb P_{\theta \setminus \kappa}$.  The forcing $\mathbb{S}_\mu * \dot{\mathbb{T}}_\mu^\mu$ is $\mu$-closed and of size $\kappa$, so it is absorbed by $\col(\mu,\kappa)$ and thus by the ideal $\p([\delta]^{<\kappa})/J$.  Again, let us use a projection that leaves a copy of $\col(\mu,\kappa)$ behind.  $G$ be generic for $\p([\delta]^{<\kappa})/J$ and let $j : V \to M$ be the generic ultrapower.  Notice that since $\square_\delta$ holds in $V$ and $\lambda = j(\kappa) = (\mu^+)^{V[G]}$, $\square_\mu$ holds in $V[G]$ by Lemma \ref{refine}.  The projected generic $g * h \subseteq \mathbb{S}_\mu * \dot{\mathbb{T}}_\mu^\mu$ yields a condition $m \in  j(\mathbb S_\mu) = \mathbb S_\mu^{V[G]}$ that is below all conditions in $g$.  Because $\square_\mu$ already holds, $\mathbb S_\mu^{V[G]}$ is in fact $\lambda$-strategically closed in $V[G]$ (but not necessarily in $M$).  Since $(2^\lambda)^M = j(\kappa^+) \leq j(\lambda)<(\lambda^+)^V$, we can use this strategic closure and the ${<}\lambda$-closure of $M$ to build an $M$-generic $\hat g$ with $m \in \hat g$.  By Theorem \ref{dualitygen}, there is an ideal $J'$ on $[\delta]^{<\kappa}$ in $V[g]$ such that $\p([\delta]^{<\kappa})/J' \cong (\p([\delta]^{<\kappa})/J)/g$.  Since $\mathbb S_\mu$ is $\kappa$-distributive, $\p([\delta]^{<\kappa})/J'$ still projects to $\col(\mu,\kappa) \times \col(\kappa,\delta) \times \add(\kappa,\lambda)$, and it is $S_\lambda$-layered.

Finally, consider the remaining forcing $\mathbb P_\mu$.  Since $\mu$ is either Mahlo or a successor, it is either $\mu$-c.c., or of the form $\mathbb P_\nu \times \mathbb S_\nu$, where $\mu = \nu^+$.   Since $\mathbb P_\mu$ has size $\mu$ and $J'$ is $\kappa$-complete, if $\bar J'$ is the ideal generated by $J'$ after forcing with $\mathbb P_\mu$, then every $\bar J'$-positive set contains a $J$-positive set from the ground model.  Thus $\p([\delta]^{<\kappa})/ \bar J'$ remains $S_\lambda$-layered, and it projects to the version of $\col(\mu,\kappa) \times \col(\kappa,\delta) \times \add(\kappa,\lambda)$ from the ground model.  If $\mu$ is Mahlo, then by Lemma \ref{newproj}, this projects to version of the same forcing as defined in the extension.  If $\mu$ is a successor, then it projects to the version as defined in the extension by $\mathbb S_\nu$, for the latter two factors because $|\mathbb S_\nu| = \mu$, and for $\col(\mu,\kappa)$ because $\mathbb S_\nu$ adds no $\nu$-sequences.  Then this version projects to the version as defined in the further extension by $\mathbb P_\nu$, because $|\mathbb P_\nu| = \nu$.  Let us record what we have done:

\begin{lemma}
\label{secondprep}
It is consistent relative to a huge cardinal that there is an inaccessible $\theta$ and a sequence $\la S_\alpha : \alpha < \theta \ra$ such that:
\begin{enumerate} 
\item $V_\theta$ satisfies GCH, $\square_\kappa$ for every infinite cardinal $\kappa$, and that there is a proper class of superstrong cardinals with Mahlo targets.
\item Clauses (\ref{statprops}) and (\ref{idealprops}) of Lemma \ref{firstprep} hold. 
\end{enumerate}
\end{lemma}

\subsection{Frequent local saturation} In order to provide the necessary set-up for the application of Radin forcing, we begin to introduce local saturation in the neighborhood of Mahlo cardinals, many of which will become singular in the end, leave room for some collapsing in between them, and retain superstrongness.

\begin{lemma}
\label{thirdprep}
Over a model satisfying the conclusion of Lemma \ref{secondprep}, there is a cofinality-preserving forcing extension in which there is an inaccessible $\theta$ and a sequence $\la S_\alpha : \alpha < \theta \ra$ such that: such that: 
\begin{enumerate}
\item \label{gchpres} $V_\theta$ satisfies GCH, $\square_\kappa$ for every infinite cardinal $\kappa$, and that there is a proper class of superstrong cardinals with Mahlo targets.
\item \label{statpres} Whenever $\mu$ is Mahlo and $\kappa = \mu^{+n}$ for $1 \leq n \leq 4$, $S_\kappa$ is a stationary subset of $\kappa \cap \cof(\mu^{+n-1})$.
\item \label{fewlocalsat} If $\mu<\theta$ is a Mahlo cardinal  and $1 \leq n \leq 3$, there is a stationary $A \subseteq \mu^{+n}$ such that $\p(A)/\ns$ is $S_{\mu^{+n+1}}$-layered.
\item \label{twomahlos} For every two Mahlo cardinals $\mu<\delta<\theta$, if $\kappa = \mu^{+4}$, and $\lambda= \delta^+$, then there is $\kappa$-complete normal ideal $I$ on $[\delta]^{<\kappa}$ such that: 
\begin{enumerate}
\item $\p([\delta]^{<\kappa})/I$ is $S_\lambda$-layered.
\item There is a stationary $A \subseteq \kappa \cap \cof(\mu^{+3})$ and a nonstationary $B \subseteq \lambda$ such that for any generic embedding $j$ arising from $I$, $\kappa \in j(A) = B$.
\item \label{projpres} $\p([\delta]^{<\kappa})/I$ projects to $\col(\mu^{+3},\kappa) \times \col(\kappa,\delta) \times \add(\kappa,\lambda)$.
\end{enumerate}
\end{enumerate}
\end{lemma}

\begin{proof}
Let $\la S_\alpha : \alpha < \theta \ra$ be the sequence given by Lemma \ref{firstprep}.  Let $\mu<\theta$ be Mahlo and let $\kappa = \mu^{+n}$ for $1 \leq n \leq 3$.  Substituting $\delta = \kappa$ in clause (\ref{idealprops}) of Lemma \ref{firstprep}, we have a normal ideal $I$ on $\kappa$ satisfying the hypotheses of Theorem \ref{localize}.  Let $A \subseteq \kappa$ be the stationary set that is forced to be mapped to a nonstationary set $B$.  By shrinking $A$ if necessary, we may assume that $S_\kappa \setminus A$ is stationary.  Let $\mathbb P_\kappa$ be the $(\kappa \cap \cof(\mu^{+n-1}) \setminus A)$-iteration as in the conclusion of Theorem \ref{localize}.  Let $X = \{ \kappa < \theta : \kappa = \mu^{+n}$ for $\mu$ Mahlo and $1 \leq n \leq 3 \}$.  We force with the product:
$$\mathbb Q := \prod^{E-\sprt}_{\kappa \in X} \mathbb P_\kappa.$$

For clause (\ref{gchpres}), for the preservation of cofinalities (and thus cardinals and squares) and the GCH, the key is to note that for any regular cardinal $\mu$, (a) $\mathbb Q \restriction \mu^+$ is $(\mu^+ \cap \cof(\mu))$-layered, and (b) $\mathbb Q \restriction [\mu^+,\theta)$ is $\mu^+$-distributive since it is of the form $(\mu^+$-distributive$)\times ( \mu^+$-closed).  Thus by Easton's Lemma, the latter retains its distributivity after forcing with $\mathbb Q \restriction \mu^+$.  The desired superstrongness is preserved using Lemma \ref{ssext}.

For clause (\ref{statpres}), we just need to check the preservation of $S_{\mu^{+n}}$ for $\mu$ Mahlo and $1 \leq n \leq 4$.  If $\mu$ is Mahlo and $1 \leq n \leq 3$, then $\mathbb P_{\mu^{+n}}$ preserves the stationarity of $S_{\mu^{+n}}$ by Lemma \ref{itershoot}, and thus so does $\mathbb Q \restriction \mu^{+n+1}$.  Since the tail $\mathbb Q \restriction [\mu^{+n+1},\theta)$ remains $\mu^{+n+1}$-distributive, $S_{\mu^{+n}}$ remains stationary.  If $n = 4$, then $\mathbb Q$ factors as $(\mu^{+n}$-c.c.$)\times(\mu^{+n}$-closed).

For clause (\ref{fewlocalsat}), let $\mu<\theta$ be Mahlo and $1 \leq n \leq 3$.  After forcing with $\mathbb P_{\mu^{+n}}$, the desired conclusion holds for some stationary $A \subseteq \mu^{+n}$, and thus it holds after forcing with $\mathbb Q \restriction [\mu^{+n},\theta)$ by the distributivity of the tail.  Temporarily let $V$ denote an extension by $\mathbb Q \restriction [\mu^{+n},\theta)$.  In this model, the forcing $\mathbb Q \restriction \mu^{+n}$ is still $(\mu^{+n} \cap \cof(\mu^{+n-1}))$-layered.  If $j : V \to M \subseteq V[G]$ is a generic embedding arising from forcing with $\p(A) / \ns$, then $M \models$ ``$j(\mathbb Q \restriction \mu^{+n})$ is $(j(\mu^{+n}) \cap \cof(\mu^{+n-1}))$-layered,'' and this holds in $V[G]$ as well since $M$ is closed under $\mu^{+n-1}$-sequences from $V[G]$.  Since $\mathbb Q \restriction \mu^{+n}$ is $\mu^{+n}$-c.c., the ideal generated by $\ns$ of the ground model is $\ns$ of the extension.  By Corollary \ref{dualitynicecase}, 
$$\mathbb Q \restriction \mu^{+n} * \p(A) / \overline{\ns} \cong \p(A) / \ns * j(\mathbb Q \restriction \mu^{+n}).$$
By Lemma \ref{iterlayer}, the right-hand side is $S_{\mu^{+n+1}}$-layered, and since $|\mathbb Q \restriction \mu^{+n}| = \mu^{+n}$, it is forced that the quotient $\p(A) / \overline{\ns}$ is $S_{\mu^{+n+1}}$-layered.

For clause (\ref{twomahlos}), let $\mu<\delta$ be Mahlo cardinals below $\theta$, and let $\kappa = \mu^{+4}$. The subforcing $\mathbb Q \restriction [\kappa,\delta)$ is $\kappa$-closed and of size $\delta$.  By the same arguments as in the previous subsection, there is an ideal $J$ on $[\delta]^{<\kappa}$ with the desired properties after forcing with $\mathbb Q \restriction [\kappa,\delta)$, and this holds after forcing further with the tail $\mathbb Q \restriction [\delta,\theta)$ by distributivity.  Now consider adjoining a generic for the forcing $\mathbb Q \restriction \kappa$.  By precisely the same argument as for (\ref{fewlocalsat}), the Booelan algebra associated to the generated ideal $\bar J$ is still $S_{\delta^+}$-layered.  Subclause (b) holds by the fact that a generic embedding arising from $\bar J$ will extend one arising from $J$, per Proposition \ref{ultequal}.  

For subclause (c), we have that 
$$\mathbb Q \restriction \kappa * \p([\delta]^{<\kappa}) / \bar{J} \cong \p([\delta]^{<\kappa}) / J * j(\mathbb Q \restriction \kappa).$$
Let $H * \bar G$ be generic for the left-hand side.  If we transfer this generic to one for the right-hand side $G * \hat H$, via the isomorphism $\iota$ of Theorem \ref{dualitygen}, we get that $G = \bar G \cap \p([\delta]^{<\kappa})^V$.  Furthermore, since $j$ is the identity on $\mathbb Q \restriction \kappa$, $H = \hat H \cap (\mathbb Q \restriction \kappa)$.  By the layeredness of $j(\mathbb Q \restriction \kappa)$, $\mathbb Q \restriction \kappa$ is a regular suborder, and $H$ is generic over $V[G]$.  Thus the map $\la q,Y \ra \mapsto \la q,\check Y \ra$ is a regular embedding of $(\mathbb Q \restriction \kappa) \times \p([\delta]^{<\kappa})/J$ into $\mathbb Q \restriction \kappa * \p([\delta]^{<\kappa}) / \bar{J}$.  Hence, if $H \subseteq \mathbb Q \restriction \kappa$ is generic over $V$, then in $V[H]$ there is a projection from $\p([\delta]^{<\kappa}) / \bar{J}$ to $[\col(\mu^{+3},\kappa) \times \col(\kappa,\delta) \times \add(\kappa,\lambda)]^V$.  This is equal to $\col(\mu^{+3},\kappa)^{V[H]} \times [ \col(\kappa,\delta) \times \add(\kappa,\lambda)]^V$.  Since $\mathbb Q \restriction \kappa$ is $\kappa$-c.c.\ and of size $\kappa$, Lemma \ref{newproj} gives that the latter factor projects to $ [ \col(\kappa,\delta) \times \add(\kappa,\lambda)]^{V[H]}$.
\end{proof}

We would like to point out that in the argument for (\ref{projpres}) above, the layeredness of $\mathbb Q \restriction \kappa$ played a substantial role.  In general, the $\kappa$-c.c.\ alone is not enough to carry out such arguments.  See \cite[Theorems 7.3 and 7.4]{MR3934477} for further discussion.

\section{The final model}
\label{radinsection}

In this section, we present a version of Radin forcing with interleaved posets and show how it forces a model of Theorem \ref{global} over a model satisfying the conclusion of Lemma \ref{thirdprep}.  We will actually have two Radin forcings, $\mathbb R$ and $\mathbb R'$, and a projection $\pi : \mathbb R' \to \mathbb R$.  $\mathbb R'$ will be slightly simpler, and will closely resemble the forcing used by Cummings in \cite{MR1041044}.  We defer to his article for some key lemmas.  $\mathbb R$ will inherit some important properties of $\mathbb R'$ and produce the desired model.

$\mathbb R$ will shoot a club through the measurable cardinals below a sufficiently strong cardinal, and between successive points $\alpha<\beta$ of this club, interleave a generic for a poset $\mathbb Q(\alpha,\beta)$ that first collapses $\beta$ to have cardinality $\alpha^{+4}$ and then iterates club shooting on $\alpha^{+4}$ to make $\ns_{\alpha^{+4}}$ locally saturated.  $\mathbb R'$ will be similar, except that it interleaves a generic for $\mathbb P(\alpha,\beta)$, a simple product of collapses and Cohen forcing which projects to $\mathbb Q(\alpha,\beta)$.  We assume the posets satisfy that following properties, which as the reader may check, will be sufficient to carry out Cummings' argument for the Prikry Property and related lemmas.
For Mahlo cardinals $\alpha<\beta$, $\mathbb P(\alpha,\beta)$, $\mathbb Q(\alpha,\beta)$, and $\pi_{\alpha,\beta}$ are such that:
\begin{enumerate}
\item $\mathbb P(\alpha,\beta)$ is a partial order that is $(2^\alpha)^{++}$-closed and of size $\leq 2^\beta$.
\item $\pi_{\alpha,\beta} : \mathbb P(\alpha,\beta) \to \mathbb Q(\alpha,\beta)$ is a projection.
\item If $\beta$ is measurable and $j : V \to M$ is an embedding derived from a normal measure on $\beta$, then:
\begin{enumerate}
\item \label{absoluteP} $\mathbb P(\alpha,\beta) = \mathbb P(\alpha,\beta)^M$.
\item There is a filter $G \subseteq \mathbb P(\beta,j(\beta))^M$ that is generic over $M$.
\end{enumerate}
\end{enumerate}

Suppose $j : V \to M$ is an elementary embedding with critical point $\kappa$ derived from an extender $E$.  Let $\mathcal U$ be the normal measure on $\kappa$ derived from $j$, and let $i_{0,1} : V \to N_1$ be the ultrapower embedding by $\mathcal U$.  Let $i_{1,2} : N_1 \to N_2$ be the ultrapower embedding of $N_1$ by $i_{0,1}(\mathcal U)$, and let $i_{0,2} = i_{1,2} \circ i_{0,1}$.   It follows from (\ref{absoluteP}) above and elementarity that $\mathbb P(\kappa,i_{0,1}(\kappa))^{N_1} = \mathbb P(\kappa,i_{0,1}(\kappa))^{N_2}$.  Let 
$$\mathbb P^* = \{ f : [\kappa]^2 \to V_\kappa : \dom f \in \mathcal U^2 \wedge \forall \alpha \forall\beta f(\alpha,\beta) \in \mathbb P(\alpha,\beta) \}.$$  
Suppose $G \subseteq \mathbb P(\kappa,i_{0,1}(\kappa))^{N_1}$ is generic over $N_1$.  Let 
$$G^* = \{ f \in \mathbb P^* : i_{0,2}(f)(\kappa,i_{0,1}(\kappa)) \in G \}.$$
Notice that $\mathcal U$ can be read off from $G^*$.  Let us define a measure sequence $u$ constructed from $(E,G)$.  Let $u(0) = \kappa$, and if $G^* \in M$, let $u(1) = G^*$.  For $\alpha > 1$, inductively let $u(\alpha) = \{ X \subseteq V_\kappa : u \restriction \alpha \in j(X) \}$, in case $u \restriction \alpha \in M$.  
We will say that $(E,G)$ is an \emph{acceptable pair} if in addition, $E$ is a $(\kappa,\lambda)$-extender, with $[\lambda]^\kappa \subseteq M$ and $|\lambda| \leq (2^\kappa)^+$.  This implies that $M^\kappa \subseteq M$ and that if $k : N_1 \to M$ is the factor embedding, then $k[G]$ generates an $M$-generic filter for $\mathbb P(\kappa,j(\kappa))^M$.

Given a measure sequence $u$ constructed by an acceptable pair $(E,G)$, we say that a set $X \subseteq V_{u(0)}$ is \emph{$u$-measure-one} if there is $f \in u(1)$ such that $\dom f = A^2$ and $A \subseteq X$, and for all $\alpha$ such that $1 < \alpha < \len u$, $X \in u(\alpha)$.
We inductively define some well-behaved classes:  Let $U_0$ be the class of $x$ such that $x$ is either a measure sequence $w$ of length $>1$ constructed by an acceptable pair, or $x=w(0)$ for such a measure sequence $w$.  Given $U_n$, let $U_{n+1}$ be the class of $x \in U_n$ such that for some measure sequence $w \in U_n$ of length $>1$, either $x = w$ or $x = w(0)$, and $U_n \cap V_{w(0)}$ is $w$-measure-one. 
Let $U_\infty = \bigcap_{n<\omega} U_n$.  If $u \in U_\infty$ and $\len u > 1$, then by countable completeness, $U_\infty \cap V_{u(0)}$ is $u$-measure-one, and $u \restriction \alpha \in U_\infty$ for $1 < \alpha < \len u$.   It is worth noting at this point:

\begin{lemma}[Cummings]
Suppose $E$ is a $(\kappa,(2^{\kappa})^+)$-extender witnessing that $\kappa$ is $(\kappa+2)$-strong, (i.e. if $j : V \to M$ is the embedding by $E$, then $V_{\kappa+2} \subseteq M$).  Let $i : V \to N$ be the embedding by the normal measure  $\mathcal U$ on $\kappa$ derived from $E$, and suppose $G$ is $\mathbb P(\kappa,i(\kappa))^N$-generic over $N$.  Then $(E,G)$ constructs a measure sequence $u \in U_\infty$ of length $(2^\kappa)^+$.
\end{lemma}
\begin{proof}See \cite[Section 3.1]{MR1041044}.
\end{proof}

Now we are ready to define the forcing $\mathbb R_u$ relative to a $u \in U_\infty$.  Let $\kappa = u(0)$.  $p$ is a condition in $\mathbb R_u$ iff $p = \la X_i : i \leq n \ra$, where $n \geq 1$ and there exists an increasing sequence of Mahlo cardinals $\kappa_0 < \dots < \kappa_n = \kappa$ such that:
\begin{enumerate}
\item $X_0 = \la \kappa_0 \ra$.
\item For $0<i<n$, $X_i$ is either a pair $\la \kappa_i,p_i \ra$ with $p_i \in \mathbb Q(\kappa_{i-1},\kappa_i)$, or a quadruple $\la w_i,A_i,H_i,h_i \ra$, where:
\begin{enumerate}
\item $w_i \in U_\infty$, $\len w_i > 1$, and $\kappa_i = w_i(0)$.
\item $A_i$ is $w_i$-measure-one and contained in $V_{\kappa_i} \setminus V_{\kappa_{i-1}}$.
\item If $(E_i,G_i)$ is an acceptable pair that constructs $w_i$, then $H_i \in G_i^*$.
\item $h_i$ is a function with domain $A_i \cap \kappa_i$, and $(\forall \alpha) h_i(\alpha) \in \mathbb P(\kappa_{i-1},\alpha)$.
\item $\dom H_i = [\dom h_i]^2$.
\end{enumerate}
\item $X_n$ is a quadruple $\la w_n,A_n,H_n,h_n \ra$ with the same properties as above, and $w_n = u$.
\end{enumerate}
If $\la X_i : i \leq n \ra$ is a condition with associated sequence of cardinals $\la \kappa_i : i \leq n \ra$, put $\kappa_{X_i} = \kappa_i$.  Let $p = \la X_i : i \leq m \ra$ and $q = \la Y_i : i \leq n \ra$.  We put $p \leq q$ when:
\begin{enumerate}
\item $m \geq n$ and $X_0 = Y_0$.
\item $\{ \kappa_{X_i} : i \leq m \} \supseteq \{ \kappa_{Y_i} : i \leq n \}$.
\item For $0<i<m$, if $X_i = \la \kappa_{X_i},p_i \ra$,
then one of the following occurs:
\begin{enumerate}
\item There is $j < n$ such that $\kappa_{X_i} = \kappa_{Y_j}$.  In this case, $Y_j$ is a pair $\la \kappa_{Y_j},q_j \ra$, $\kappa_{X_{i-1}} =  \kappa_{Y_{j-1}}$ also, and $p_i \leq q_j$.
\item \label{projpoints} There is no $j < n$ such that $\kappa_{X_i} = \kappa_{Y_j}$.  For the least $j\leq n$ such that $\kappa_{X_{i}} <  \kappa_{Y_j}$, $Y_j$ is a quadruple $\la w,A,H,h \ra$ with $\kappa_{X_i} \in A$.  If $\kappa_{X_{i-1}} = \kappa_{Y_{j-1}}$, then $p_i \leq \pi_{\kappa_{X_{i-1}},\kappa_{X_{i}}}(h(\kappa_{X_i}))$, and if $\kappa_{X_{i-1}} > \kappa_{Y_{j-1}}$, then $p_i \leq \pi_{\kappa_{X_{i-1}},\kappa_{X_{i}}}(H(\kappa_{X_{i-1}},\kappa_{X_i}))$.
\end{enumerate}
\item For $0<i \leq m$, if $X_i$ is a quadruple $\la w,A,H,h \ra$,  then one of the following occurs:
\begin{enumerate}
\item There is $j \leq n$ such that $Y_j = \la w,A',H',h' \ra$.  In this case, $A \subseteq A'$ and for all $(\alpha,\beta) \in \dom H$, $H(\alpha,\beta) \leq H'(\alpha,\beta)$.  If $\kappa_{X_{i-1}} = \kappa_{Y_{j-1}}$, then for all $\alpha \in \dom h$, $h(\alpha) \leq h'(\alpha)$.  If $\kappa_{X_{i-1}} > \kappa_{Y_{j-1}}$, then for all $\alpha \in \dom h$, $h(\alpha) \leq H'(\kappa_{X_{i-1}},\alpha)$.
\item There is no $j \leq n$ such that $\kappa_{X_i} = \kappa_{Y_j}$, and for the least $j \leq n$ such that $\kappa_{X_i} < \kappa_{Y_j}$, $Y_j$ is a quadruple $\la v,A',H',h' \ra$ such that $w \in A'$, $A \subseteq A'$, and $H(\alpha,\beta) \leq H'(\alpha,\beta)$ for all $(\alpha,\beta) \in \dom H$.   If $\kappa_{X_{i-1}} = \kappa_{Y_{j-1}}$, then for all $\alpha \in \dom h$, $h(\alpha) \leq h'(\alpha)$.  If $\kappa_{X_{i-1}} > \kappa_{Y_{j-1}}$, then for all $\alpha \in \dom h$, $h(\alpha) \leq H'(\kappa_{X_{i-1}},\alpha)$. 
\end{enumerate}
\end{enumerate}

Now the forcing $\mathbb R'_u$ is the same except that for every $\alpha<\beta$ Mahlo, we replace $\mathbb Q(\alpha,\beta)$ with $\mathbb P(\alpha,\beta)$ and $\pi_{\alpha,\beta}$ with the identity function.  The forcings $\mathbb R'_u$ are of the type studied in \cite{MR1041044}.  In our more general class of forcings $\mathbb R_u$, the simpler posets $\mathbb P(\alpha,\beta)$ guide us along, until we have decided that $\alpha<\beta$ are  successive points of the Radin sequence, and then we project to $\mathbb Q(\alpha,\beta)$.

We say $p \leq^* q$ when $p \leq q$ and $\len p = \len q$.  It is easy to see that if $u \in U_\infty$, $p = \la \la \kappa_0 \ra, X_1 \ra$, then $\la \mathbb R'_u \restriction p, \leq^* \ra$ is $(2^{\kappa_0})^{++}$-closed, and $\la \mathbb R_u \restriction p, \leq^* \ra$ is $(2^{\kappa_0})^{++}$-strategically closed.

\begin{lemma}
\label{prikryproj}
If $u \in U_\infty$, then there is a length-preserving projection $\pi : \mathbb R_u' \to \mathbb R_u$.  Moreover, if $q \leq \pi(p)$, then there is $p' \leq p$ such that $\pi(p') \leq^* q$.
\end{lemma}

\begin{proof}
Suppose $p = \la X_i : i \leq n \ra \in \mathbb R_u'$.  Let $\pi(p) = \la Y_i : i \leq n \ra$, where $X_i = Y_i$ if $i = 0$ or $X_i$ is a quadruple, and if $X_i = \la \kappa_i,p_i \ra$, then $Y_i = \la \kappa_i, \pi_{\kappa_{i-1},\kappa_i}(p_i) \ra$.  $\pi$ is order-preserving because each $\pi_{\alpha,\beta}$ is, and by requirement (\ref{projpoints}) in the definition of the ordering for $\mathbb R_u$.

Suppose $p = \la X_i : i \leq n \ra \in \mathbb R'_u$, $q = \la Y_i : i \leq m \ra \in \mathbb R_u$, and $q \leq \pi(p)$.  We need to find a condition $p' = \la Z_i : i \leq m \ra \leq p$ such that $\pi(p') \leq q$. If $i = 0$ or $Y_i$ is a quadruple, let $Z_i = Y_i$.   Suppose $i<m$ is such that $Y_i$ is a pair $\la \kappa,q_i \ra$, and let $\mu = \kappa_{Y_{i-1}}$.  If there is $j < n$ such that $\kappa = \kappa_{X_j}$, then $X_j$ is a pair $\la \kappa,p_j \ra$, $\mu = \kappa_{X_{j-1}}$, and $q_i \leq \pi_{\mu,\kappa}(p_j)$.  Find $p'_i \leq p_j$ such that $\pi_{\mu,\kappa}(p'_i) \leq q_i$, and put $Z_i = \la \kappa,p'_i \ra$.
If there is no such $j<n$, then let $j \leq n$ be least such that $\kappa_{Y_i} < \kappa_{X_j}$ and $X_j$ is a quadruple $\la w,A,H,h \ra$.  If $\mu = \kappa_{X_{j-1}}$, let $p_i = h(\kappa)$, and otherwise let $p_i = H(\mu,\kappa)$.
We must have that $q_i \leq \pi_{\mu,\kappa}(p_i)$. Find $p'_i \leq p_i$ such that $\pi_{\mu,\kappa}(p'_i) \leq q_i$, and put $Z_i = \la \kappa,p'_i \ra$.
It is straightforward to check that $p'$ is as desired.
\end{proof}

The following two lemmas are easy to verify:

\begin{lemma}
\label{radinfactor}
Suppose $u \in U_\infty$ and $p = \la X_i : i \leq n \ra \in \mathbb R_u$.  Suppose $m_0<m_1< n$ are such that $X_{m_0}$ is a quadruple $\la w,A,H,h \ra$, and $X_i$ is a pair $\la \kappa_i,p_i \ra$ for $m_0 < i \leq m_1$.  Then $\mathbb R_u \restriction p$ is isomorphic to
$$ \mathbb R_{w} \restriction \la X_i : i \leq m_0 \ra \times \mathbb Q(\kappa_{m_0}, \kappa_{m_0+1}) \restriction p_{m_0+1} \times \dots \times \mathbb Q(\kappa_{m_1-1}, \kappa_{m_1})  \restriction p_{m_1}$$
$$\times \mathbb R_u \restriction \la \la \kappa_{m_1} \ra, X_{m_1+1},\dots,X_n \ra.$$
\end{lemma}

\begin{lemma}
Suppose $u \in U_\infty$ and $G \subseteq \mathbb R_u$ is generic.  Let $\kappa = u(0)$ and let $C = \{ \alpha : (\exists p = \la X_i : i \leq n \ra \in G)(\exists i < n) \alpha = \kappa_{X_i} \}$.  Then $C$ is club in $\kappa$.
\end{lemma}

The most important result concerning these forcings is the following, known as the Prikry Property.  The proof takes some work, and we refer the reader to \cite[Section 3.4]{MR1041044} for a complete account.

\begin{theorem}[Cummings]
\label{prikry}
Suppose $u \in U_\infty$, $p \in \mathbb R'_u$, and $\sigma$ is a sentence in the forcing language of $\mathbb R'_u$.  Then there is $q \leq^* p$ deciding $\sigma$.
\end{theorem}

\begin{corollary}
\label{prikrycor}
The Prikry Property also holds for $\mathbb R_u$:
If $u \in U_\infty$, $p \in \mathbb R_u$, and $\sigma$ is a sentence in the forcing language of $\mathbb R_u$, then there is $q \leq^* p$ deciding $\sigma$. 
\end{corollary}

\begin{proof}
Let $\pi : \mathbb R_u' \to \mathbb R_u$ be the projection given by Lemma \ref{prikryproj}.  Let $q_0 \in \mathbb R_u$ and let $\sigma$ be a sentence in the forcing language of $\mathbb R_u$.  Let $p_0 \in \mathbb R'$ be such that $\pi(p_0) \leq^* q_0$.  Let $p_1 \leq^* p_0$ decide whether $\sigma$ holds in the submodel $V^{\mathbb R_u}$, 
and let us assume it forces that $\sigma$ holds.  If $\pi(p_1)$ does not force $\sigma$, let $q_1 \leq \pi(p_1)$ force $\neg \sigma$.  Let $p_2 \leq p_1$ be such that $\pi(p_2) \leq q_1$.  But then $p_2$ forces both that $\sigma$ and $\neg\sigma$ hold in the submodel $V^{\mathbb R_u}$, a contradiction.  Thus $\pi(p_1) \leq^* q_0$ and $\pi(p_1)$ decides $\sigma$.
\end{proof}

\begin{corollary}
\label{radincards}
Suppose $u \in U_\infty$, $p = \la X_i : i \leq n \ra \in \mathbb R_u$, and $X_0 = \la \kappa_0 \ra$.  Then $\mathbb R_u \restriction p$ adds no subsets of $(2^{\kappa_0})^+$.  Thus if $\mathbb P$ is a forcing of size $\leq (2^{\kappa_0})^+$, then $\mathbb R_u$ adds no subsets of $(2^{\kappa_0})^+$ over $V^{\mathbb P}$.
\end{corollary}

Now we specify the partial orders $\mathbb P(\alpha,\beta)$ and $\mathbb Q(\alpha,\beta)$ and the projections $\pi_{\alpha,\beta}$.  For Mahlo $\alpha<\beta$, let
$$\mathbb P(\alpha,\beta) = \col(\alpha^{+4},\beta) \times \col(\alpha^{+3},\alpha^{+4}) \times \add(\alpha^{+4},\beta^+).$$
In a model satisfying the conclusion of Lemma \ref{thirdprep}, we have by Lemma \ref{collapsedelta} that whenever $\alpha < \beta$ are Mahlo, $\col(\alpha^{+4},\beta)$ forces that there is a saturated ideal on $\alpha^{+4}$ whose associated forcing projects to $\col(\alpha^{+3},\alpha^{+4}) \times \add(\alpha^{+4},\beta^+)$.  By Lemma \ref{localize}, there is in this model a stationary $S \subseteq S_{\alpha^{+4}}$ (the latter being the one that witnesses layeredness of the ideal on $\alpha^{+3}$), and there is an $S$-iteration $\mathbb C_{\alpha^{+4}}$ of length $\beta^+$ that forces $\ns_{\alpha^{+4}}$ to be locally saturated.  By Lemma \ref{absorbshoot}, there is a projection 
$$\sigma_{\alpha,\beta} : \col(\alpha^{+3},\alpha^{+4}) \times \add(\alpha^{+4},\beta^+) \to \mathbb C_{\alpha^{+4}}.$$  Thus we let $\mathbb Q(\alpha,\beta) = \col(\alpha^{+4},\beta) * \dot{\mathbb C}_{\alpha^{+4}}$, and we let
$\pi_{\alpha,\beta}(p,q,r) = \la p,\dot\sigma_{\alpha,\beta}( \check q,\check r ) \ra.$
These specifications clearly fulfill the first two requirements for our partial orders.  For the third, if $\beta$ is measurable, $\alpha<\beta$ is Mahlo, and $j : V \to M$ is an embedding derived from a normal measure on $\beta$, then $\mathbb P(\alpha,\beta) = \mathbb P(\alpha,\beta)^M$ because $M$ correctly computes $\beta^+$.  By GCH, $\beta^+ <j(\beta^+) <\beta^{++}$.  Thus by the $\beta^+$-closure and $j(\beta^+)$-c.c.\ in $M$ of $\mathbb P(\beta,j(\beta))^M$, we can build filter $G \in V$ that is generic over $M$.


For the remainder of this article, when we refer to the forcings $\mathbb R_u$ and $\mathbb R_u'$, we mean those defined in our preparatory model using the above specifications.  The important feature of our version of Radin forcing, which is not shared by Cummings' version, is the chain condition:

\begin{lemma}
\label{radincc}
Suppose $u \in U_\infty$ and $\kappa = u(0)$.  Then $\mathbb R_u$ is $\kappa^+$-c.c.  Moreover, it preserves $\kappa^{++}$-saturated ideals on $\kappa^+$.
\end{lemma}
\begin{proof}
Suppose $\la p_\alpha : \alpha < \kappa^+ \ra \subseteq \mathbb R_u$.  Let $p_\alpha = \vec x_\alpha \,^\frown \la u,A_\alpha,H_\alpha,h_\alpha \ra$.  We can assume there is a fixed $\vec x =  \la X_0,\dots,X_{n-1}\ra$ such that $\vec x_\alpha = \vec x$ for all $\alpha$.  Let $\mathcal U$ be the ultrafilter associated to $u(1)$ and let $j_{\mathcal U} : V \to M$ be the embedding by $\mathcal U$.  For each $\alpha$, $[h_\alpha]_{\mathcal U} \in \mathbb P(\kappa_{n-1},\kappa)$, which is $(\kappa^+ \cap \cof(\kappa))$-layered.  Let $\alpha < \beta$ be such that $h_\alpha,h_\beta$ represent compatible conditions.  Since $H_\alpha,H_\beta$ represent conditions in a filter, $p_\alpha$ and $p_\beta$ are compatible.

To show that $\mathbb R_u$ preserves saturated ideals on $\kappa^+$, suppose $I$ is such an ideal and $i : V \to N \subseteq V[G]$ is a generic ultrapower via $I$.  Then the above argument can be carried out in $N$.  In particular, if $M'$ is the ultrapower of $N$ by $i(\mathcal U)$, then $N$ satisfies that $\mathbb P(\kappa_{n-1},\kappa)$ is $(i(\kappa^+) \cap \cof(\kappa))$-layered.  This is true in $V[G]$ as well since $N^\kappa \cap V[G] \subseteq N$.  Thus by Corollary \ref{dualitynicecase}, $\mathbb R_u$ forces that the ideal generated by $I$ is $\kappa^{++}$-saturated.
\end{proof}

\begin{lemma}[Cummings]
\label{measpresspecial}
Suppose $u \in U_\infty$, $u(0) = \kappa$, and $\len u \geq (2^{\kappa})^+$.  Then $\mathbb R'_u$ preserves the measurability of $\kappa$.
\end{lemma}

\begin{proof}
See \cite[Section 3.8]{MR1041044}.
\end{proof}

Now fix a cardinal $\kappa$ which is $(\kappa+2)$-strong as witnessed by a $(\kappa,\kappa^{++})$-extender $E$, let $G$ be such that $(E,G)$ is an acceptable pair, and let $u$ be a measure sequence of length $\kappa^{++}$ constructed by $(E,G)$.  Let $H \subseteq \mathbb R_u$ be generic.  By Lemma \ref{measpresspecial}, $V_\kappa^{V[H]}$ is a model of ZFC.  Let $C \subseteq \kappa$ be the Radin club introduced by $H$, and let $\kappa_0 = \min C$.  Because of Lemma \ref{radinfactor}, Corollary \ref{radincards}, and the interleaved collapses, the set of limit cardinals in $\kappa \setminus \kappa_0$ in $V[H]$ is simply the set of limit points of $C$.  Let $h \subseteq \col(\omega,\kappa_0)$ be generic over $V[H]$.  We claim that $V_\kappa^{V[H][h]}$ is a model demonstrating Theorem \ref{global}.

First we check that $\square_\delta$ holds for all $\delta < \kappa$.  If $\mu$ is a successor cardinal of $V[H][h]$, then $\mu = \nu^+$ for some cardinal $\nu$ of $V$.  $\square_\nu$ holds in $V$.  Although $\nu$ may be collapsed, Lemma \ref{refine} implies that $\square_\eta$ holds in $V[H][h]$, where $\eta$ is the predecessor of $\mu$ in the final model.


Now let $\mu_0 < \mu_1$ be two successive points of $C$.   Suppose first that $\mu_0$ is a limit point.  Let $p \in H$ force this, so that $p = \la X_i : i \leq n \ra$ is such that for some $m < n$, $X_m$ is a quadruple $\la w,A,H,h \ra$ with $w(0) = \mu_0$ and $X_{m+1}$ is a pair $\la \mu_1,p_{m+1} \ra$  By Lemma \ref{radinfactor}, 
$$\mathbb R_u \restriction p  \cong \mathbb R_w \restriction \la X_0,\dots,X_m \ra \times \mathbb Q(\mu_0,\mu_1) \restriction p_{m+1} \times \mathbb R_u \restriction \la \la \mu_1 \ra, X_{m+2},\dots,X_n \ra.$$

In $V$, $\ns_{\mu_0^+}$ is locally saturated, and this is preserved by $\mathbb R_w$ by Lemma \ref{radincc}.  The local saturation of $\ns_{\mu_0^{++}}$ is preserved since $|\mathbb R_w | = \mu_0^+$.  The upper factor $\mathbb Q(\mu_0,\mu_1) \times \mathbb R_u \restriction \la \la \mu_1 \ra, X_{m+2},\dots,X_n \ra$ preserves this, since it adds no further subsets to $\mu_0^{+3}$.

For $\mu_0^{+3}$, in $V$ there is some stationary $A \subseteq \mu_0^{+3}$ such that $\p(A)/\ns$ is $S_{\mu_0^{+4}}$-layered.  The factor $\mathbb Q(\mu_0,\mu_1)$ preserves the stationarity of $S_{\mu_0^{+4}}$ and adds no subsets of $\mu_0^{+3}$, and thus preserves that $\ns_{\mu_0^{+3}}$ is locally saturated.  This is preserved by the small lower factor $\mathbb R_w$, and by the upper factor $\mathbb R_u \restriction \la \la \mu_1 \ra, X_{m+2},\dots,X_n \ra$, which adds no further subsets of $\mu_0^{+4}$.

For $\mu_0^{+4}$, the local saturation of $\ns_{\mu_0^{+4}}$ is explicitly forced by $\mathbb Q(\mu_0,\mu_1)$.  This is preserved by the small lower factor $\mathbb R_w$, and then by the upper factor $\mathbb R_u \restriction \la \la \mu_1 \ra, X_{m+2},\dots,X_n \ra$, which adds no further subsets of $\mu_1^+ = (\mu_0^{+5})^{V[H]}$.

Now suppose that $\mu_0$ is a successor point  or the least point of $C$.  If $p \in H$ forces this, then we may assume that $p = \la X_i : i \leq n \ra$ and 
$$\mathbb R_u \restriction p  \cong \mathbb P \times \mathbb Q(\mu_0,\mu_1) \restriction p_{m+1} \times \mathbb R_u \restriction \la \la \mu_1 \ra, X_{m+2},\dots,X_n \ra,$$
where $\mathbb P$ is either trivial or $(\mu_0^+ \cap \cof(\mu_0))$-layered.  $\mathbb P$ preserves the local saturation of $\ns_{\mu_0^+}$, and this is preserved by the upper factor, which adds no subsets of $\mu_0^{++}$.  The local saturation of $\ns_{\mu_0^{+k}}$ for $2 \leq k \leq 4$ is forced for the same reasons as in the case that $\mu_0$ is a limit point.  

Finally, all of these saturation properties are preserved by the small forcing $\col(\omega,\kappa_0)$, which makes the class of successor cardinals in $V[H]$ above $\kappa_0$ equal to the class of all successor cardinals.
This concludes the proof of Theorem \ref{global}. 


 
\bibliographystyle{amsplain.bst}
\bibliography{localsatsquare.bib}

\end{document}